\documentclass[12pt]{amsart}

\usepackage{amscd}
\usepackage{amsmath,graphicx}
\usepackage{verbatim}

\usepackage{amssymb}
\usepackage{pb-diagram,pb-xy}
\usepackage{ bbold }

\usepackage{amssymb}

\usepackage[thinc]{esdiff}

\usepackage{mathtools}

\usepackage{xy}
\xyoption{all}

\numberwithin{equation}{section}

\DeclareMathAlphabet{\cat}{OT1}{cmss}{m}{sl}

\newtheorem{theorem}[equation]{Theorem}
\newtheorem{proposition}[equation]{Proposition}
\newtheorem{lemma}[equation]{Lemma}
\newtheorem{corollary}[equation]{Corollary}

\newtheorem*{itheorem}{Theorem}

\theoremstyle{definition}
\newtheorem{remark}[equation]{Remark}
\newtheorem{example}[equation]{Example}
\newtheorem{definition}[equation]{Definition}
\newtheorem{observation}[equation]{Observation}

\newtheorem{notation}[equation]{Notation}

\renewcommand{\(}{\bigl(}
\renewcommand{\)}{\bigr)}

\newcommand{\tens}{\otimes}
\newcommand{\inv}{^{-1}}

\newcommand{\iso}{\stackrel{\sim}{\to}}

\newcommand{\id}{\mathrm{id}}

\newcommand{\Op}{\operatorname{\bold{OP}}}
\newcommand{\CK}{\operatorname{CK}}

\renewcommand{\Im}{\operatorname{Im}}
\newcommand{\Ad}{\operatorname{Ad}}

\newcommand{\Ker}{\operatorname{Ker}}
\newcommand{\rank}{\operatorname{rank}}

\newcommand{\Colim}{\operatorname{colim}}

\newcommand{\Spec}{\operatorname{Spec}}

\newcommand{\Hom}{\operatorname{Hom}}

\newcommand{\xra}{\xrightarrow}

\newcommand{\Coprod}{\operatornamewithlimits{\textstyle\coprod}}

\newcommand{\Bigcap}{\operatornamewithlimits{\textstyle\bigcap}}

\newcommand{\Sum}{\operatornamewithlimits{\textstyle\sum}}
\newcommand{\Prod}{\operatornamewithlimits{\textstyle\prod}}

\DeclarePairedDelimiter\floor{\lfloor}{\rfloor}

\newcommand{\mybinom}[3][0.8]{\scalebox{#1}{$\dbinom{#2}{#3}$}}

\renewcommand{\P}{\mathbb{P}}
\newcommand{\Z}{\mathbb{Z}}

\newcommand{\Q}{\mathbb{Q}}
\newcommand{\R}{\mathbb{R}}

\newcommand{\cL}{\mathcal L}

\newcommand{\cO}{\mathcal O}

\newcommand{\cP}{\mathcal P}
\newcommand{\cQ}{\mathcal Q}

\newcommand{\zz}{{\Bbb Z}}

\newcommand{\nn}{\Bbb N}

\newcommand{\qq}{{\Bbb Q}}
\newcommand{\pp}{{\Bbb P}}

\newcommand{\spec}{\operatorname{Spec}}

\newcommand{\op}[1]{\operatorname{#1}}

\newcommand{\ffi}{\varphi}

\newcommand{\eps}{\varepsilon}

\newcommand{\row}{\rightarrow}

\newcommand{\lrow}{\longrightarrow}
\renewcommand{\leq}{\leqslant}
\renewcommand{\geq}{\geqslant}

\newcommand{\nichego}[1]{}

\newcommand{\ov}[1]{\overline{#1}}

\newcommand{\wt}[1]{\widetilde{#1}}

\newcommand{\smk}{{\mathbf{Sm}}_k}

\newcommand{\smop}{{\mathbf{SmOp}}}

\newcommand{\CH}{\operatorname{CH}}

\newcommand{\wzz}{\widehat{\zz}}
\newcommand{\ppsi}[3]{\prescript{#1}{}{\Psi}^{#2}_{#3}}

\title
[Operations in connective $K$-theory] 
{Operations in connective $K$-theory}

\author[A.~Merkurjev]
{Alexander Merkurjev}

\address{Alexander Merkurjev\\
        Department of Mathematics \\
        University of California\\
        Los Angeles, CA \\
        USA}

\email
{merkurev@math.ucla.edu}

\author[A.~Vishik]
{Alexander Vishik}

\address{Alexander Vishik\\
        School of Mathematical Sciences \\
        University of Nottingham\\
        University Park\\
        Nottingham, NG7 2RD \\
        United Kingdom}

\email
{alexander.vishik@nottingham.ac.uk}

\begin{document}

\begin{abstract}
In this article we classify additive operations in connective
K-theory with various torsion-free coefficients.
We discover that the answer for the integral case requires
understanding of the $\widehat{\zz}$ one. Moreover, although
integral additive operations are topologically generated by Adams operations, these are not reduced to
infinite linear combinations of the latter ones.
We describe a topological basis for stable operations and relate it to a
basis of stable operations in graded K-theory. We classify multiplicative operations in both theories
and show that homogeneous additive stable operations with
$\widehat{\zz}$-coefficients are topologically generated
by stable multiplicative operations. This is not true for
integral operations.
\end{abstract}

\thanks{The first author has been supported by the NSF grant DMS \#1801530.}

\maketitle

\section{Introduction}

Let $k$ be a field of characteristic $0$. An \emph{oriented cohomology theory} $A^*$ over $k$ is a
functor from the category $\smk^{op}$ of smooth quasi-projective varieties over $k$
to the category of $\Z$-graded commutative rings equipped with a push-forward
structure and satisfying certain axioms. In this article, we study the, so-called,
{\it small theories}. For these, the appropriate choice is \cite[Definition 2.1]{vishik19}
which employs a strong form of
the localisation axiom and is some breed of the axioms of
Panin-Smirnov \cite[Definition 1.1.7]{Panin04} and that of Levine-Morel
\cite[Definition 1.1.2]{LM07}.
In particular, every oriented cohomology theory $A^*$ admits a theory of Chern classes $c^A_n$ of vector bundles.
Among such theories there is the universal one - the algebraic cobordism of Levine-Morel $\Omega^*$ \cite{LM07}.
We will work with the {\it free} theories, i.e. theories obtained from $\Omega^*$ by change of coefficients. These are
exactly the theories of {\it rational type} for which the results of
\cite{vishik19} apply.

Examples of free oriented cohomology theories are:

\smallskip

\noindent $\bullet$ \emph{Chow theory}
$\CH^*$ that assigns to a smooth variety $X$ over $k$ the Chow ring $\CH^*(X)$;

\noindent $\bullet$ \emph{Graded $K$-theory} $K_{gr}^*$ taking $X$ to the Laurent polynomial ring $K_0(X)[t,t\inv]$
(graded by the powers of the \emph{Bott element} $t$
of degree $-1$) over the Grothendieck ring $K_0(X)$;

\noindent $\bullet$ \emph{Connective $K$-theory} taking a smooth variety $X$ to the ring $\CK^*(X)$ of $X$ (see \cite{Cai08} and \cite{DL2014}).

The connective $K$-theory is the ``smallest" oriented cohomology theory ``living" above Chow theory and
graded $K$-theory: there are natural graded morphisms
\[
\xymatrix{
 & \CK^*(X)  \ar[ld] \ar[rd]  \\
\CH^*(X)   & &  K_{gr}^*(X)
}
\]
that yield graded isomorphisms
\[
\CK^*(X)/t\CK^{*+1}(X)\iso \CH^*(X)\quad\text{and}\quad \CK^*(X)[t\inv]\iso K_{gr}^*(X).
\]
Moreover, multiplication $\CK^{n+1}(X)\xra{t}\CK^{n}(X)$ by the Bott element $t\in\CK^{-1}(k)$
is an isomorphism if $n<0$.
The map $\CK^{0}(X)\to K_{gr}^0(X)=K_0(X)$ is also an isomorphism, so we can identify $\CK^{n}(X)$ with $K_0(X)$
for all $n\leq 0$.

For any $n\geq 0$ the image of $\CK^n(X)\xra{t^n} \CK^0(X)=K_0(X)$ is the subgroup $K_0^{(n)}(X)\subset K_0(X)$ generated by
the classes of coherent $\cO_X$-modules with codimension of support at least $n$. Note that the map $t^n$ may not be
injective in general if $n>1$.

\smallskip

Let $A^*$ and $B^*$ be two oriented cohomology theories. An \emph{additive operation}
$G:A^*\to B^*$ is a morphism between functors $A^*$ and $B^*$ considered as contravariant functors from
$\smk$ to the category of abelian groups. Examples of additive operations are Adams
operations in algebraic $K$-theory and Steenrod operations in the Chow groups modulo a prime integer.

If $A^*$ is an oriented cohomology theory and $R$ is a commutative ring, we write
$A^{n}_R(X)$ for $A^{n}(X)\tens_{\Z} R$ and $\Op_R^{n,m}(A)$ for the $R$-module of $R$-linear operations $A^n_R\to A^m_R$.

It is proved in \cite[\S 6.3]{vishik19} that every free
oriented cohomology theory $A^*$ admits the \emph{Adams operations}
$\Psi^A_m\in \Op_R^{n,n}(A)$
for all $n$ and $m$. The operation $\Psi^A_m$ in $\Op_R^{1,1}(A)$ satisfies
\[
\Psi^A_m(c_1^{A}(L))=c_1^{A}(L^{\tens m})
\]
for a line bundle $L$.
Moreover, there is an $R$-linear map
\[
\Ad_n:R[[x]]\to \Op_R^{n,n}(A)
\]
taking the power series $(1-x)^m$ to the
Adams operation $\Psi^A_m$ for all $m\in\Z$.

In general, the map $\Ad_n$ is neither injective nor surjective - see below. But it is shown in \cite[\S 6.1]{vishik19}
that $\Ad_n$ is an isomorphism if $A^*$ is the graded $K$-theory, thus,
\[
\Op_R^{n,n}(K_{gr})\simeq R[[x]].
\]
Since the power series $(1-x)^m$ generate $R[[x]]$ as topological $R$-module in the $x$-adic topology,
we can say that the $R$-module $\Op_R^{n,n}(K_{gr})$
is topologically generated by the Adams operations in the graded $K$-theory. Moreover, since multiplication by the Bott
element is an isomorphism in $K_{gr}^*$, we have $\Op_R^{n,m}(K_{gr})= R[[x]]\cdot t^{n-m}$.

\medskip
In the present paper we study the groups $\Op_R^{n,m}:=\Op_R^{n,m}(\CK)$
of operations in the connective $K$-theory over $R$. We write for simplicity $\Op^{n,m}$ for $\Op_{\Z}^{n,m}$.

The groups $\CK^n(X)$ for $n\leq 0$ are identified with $K_0(X)$, hence translating the above result on the
operations in graded $K$-theory, we see that $\Ad_n:R[[x]]\to \Op_R^{n,n}$ is an isomorphism for $n\leq 0$.

The Adams operation $\Psi_0$ is trivial on $\CK_R^n$ for $n\geq 1$, i.e. $\Ad_n(1)=0$, so we consider the restriction
$\Ad'_n:xR[[x]]\to  \Op_R^{n,n}$ of the map $\Ad_n$. The $R$-module $\CK_R^1(X)$ is a canonical direct summand via multiplication by $t$ of $\CK_R^0(X)=K_0(X)_R$
with the complement $R\cdot 1$. This leads to a ring isomorphism $\Op_R^{0,0}\simeq R\times \Op_R^{1,1}$. Moreover,
the map $\Ad'_1:xR[[x]]\to  \Op_R^{1,1}$ is an isomorphism.

The structure of the groups $\Op_R^{n,n}$ with $n>1$ is much more delicate and depends on the base ring $R$.
The homomorphisms $\Ad'_n:xR[[x]]\to  \Op_R^{n,n}$ for $n\geq 2$ are not surjective in general.

It came as a surprise to us that the structure of $\Op_R^{n,n}$ is very simple over
the ring of \emph{profinite integers} $\widehat \Z=\lim(\Z/n\Z)$:

\begin{itheorem}
The map $\Ad'_n:x\widehat \Z[[x]]\to  \Op_{\widehat \Z}^{n,n}$ is an isomorphism if $n\geq 1$. In particular,
the $\widehat \Z$-module $\Op_{\widehat \Z}^{n,n}$ is topologically generated by the Adams operations.
\end{itheorem}

Over $\Z$ the map $\Ad'_n$ is not surjective if $n\geq 2$.

\begin{itheorem}
The group $\Op^{n,n}$ of integral operations is isomorphic canonically to a subgroup of $\Op^{n,n}_{\widehat \Z}$.
Moreover, there is an exact sequence
\[
0\to x\Z[[x]]\xra{\Ad'_n} \Op^{n,n} \to (\widehat\Z/\Z)^{n-1} \to 0
\]
if $n\geq 1$.
\end{itheorem}

Thus, the group $\widehat \Z$ also shows up in the computation of
$\Op^{n,n}$ over $\Z$.
For example, $\Op^{2,2}$ as a subgroup of $\Op^{2,2}_{\widehat \Z}=x\widehat \Z[[x]]$ is generated by
$x\Z[[x]]$ and the power series $\sum_{i>0}\frac{c-c_i}{i}\ x^i$ for all $c\in \widehat \Z$ and integers $c_i$ such that
$c-c_i$ is divisible by $i$ for all $i>0$, i.e., $c_i$ in $\Z$ represents congruence class of $c$ modulo $i$.

\smallskip

We prove that the rings $\Op^{n,n}$ and $\Op^{n,n}_{\widehat\Z}$ are commutative. Moreover, the rings $\Op^{n,n}$
are ``almost" integral domains: the only zero divisors are the multiples of $\Psi_{1}\pm \Psi_{-1}$.

\smallskip

An operation $G:A^*\row B^*$ is called \emph{multiplicative} if $G$ is a morphism of functors $\smk\row Rings$.
Examples are \emph{twisted} Adams operations $\Psi_b^c$ defined as follows. Let $b\in\widehat\Z$ and $c\in\widehat\Z^\times$.
Then the operation $\Psi_b^c$ is homogeneous and equal to $c^{-n}\cdot\Psi_{bc}$ on $\CK^n_{\widehat\Z}$, where
$\Psi_{bc}$ is the (generalized) Adams operation with the power series $(1-x)^{bc}$. We classify all
multiplicative operations on $\CK^*_{\widehat\Z}$ in Section \ref{multoperat}.

The notion of ``stability" in topology can be considered in algebraic setting as follows (see \cite[\S 3.1]{vishik19}).
Let $\smop$ be a category whose objects are pairs
$(X,U)$, where $X\in\smk$ and $U$ is an open subvariety of $X$.
Any theory $A^*$ extends from $\smk$ to $\smop$ by the rule:
$$
A^*((X,U)):=\op{Ker}(A^*(X)\row A^*(U)).
$$
and every additive operation $A^*\row B^*$ on $\smk$ extends uniquely to an operation on $\smop$.
There is an identification
$$
\sigma_T^A:A^*((X,U))\stackrel{\cong}{\lrow} A^{*+1}(\Sigma_T(X,U)),
$$
where $\Sigma_T(X,U):=(X,U)\wedge(\pp^1,\pp^1\backslash 0)=(X\times\pp^1, X\times(\pp^1\backslash 0)\cup U\times\pp^1)$.

For any additive operation $G:A^*\row B^*$ we define its \emph{desuspension}
as the unique operation $\Sigma^{-1} G:A^*\row B^*$ such that
$$
G\circ\sigma_T^A=\sigma_T^B\circ\Sigma^{-1}G.
$$
A \emph{stable} additive operation $G:A^*\row B^*$ is the collection
$\{G^{(n)}|n\geq 0\}$ of operations $A^*\row B^*$ such that
$G^{(n)}=\Sigma^{-1}G^{(n+1)}$.

In Section \ref{stableoper} we classify stable operations in connective $K$-theory over $\widehat \Z$. We prove that under
the identification
\[
\Op_{\widehat \Z}^{n,n}=
\left\{
  \begin{array}{ll}
   \widehat \Z[[x]], & \hbox{if $n\leq 0$;} \\
   x\widehat \Z[[x]], & \hbox{if $n\geq 1$}
  \end{array}
\right.
\]
the desuspension map is given by the formula
\[
\Sigma^{-1}(G)=
\left\{
  \begin{array}{ll}
   \Phi(G), & \hbox{if $n\leq 1$;} \\
   \Phi(G)-\Phi(G)(0), & \hbox{if $n>1$.}
  \end{array}
\right.
\]
where $G\in \Op_{\widehat \Z}^{n,n}$ and $\Phi(G)=(x-1)\diff{G}{x}$.
Thus, the desuspension map $\Sigma^{-1}$ yields a tower of injective maps
\[
\widehat\Z[[x]]=\Op_{\widehat\Z}^{0,0}\hookleftarrow \Op_{\widehat\Z}^{1,1}\hookleftarrow \ldots\hookleftarrow \Op_{\widehat\Z}^{n,n}\hookleftarrow \ldots.
\]
The group of homogeneous
degree $0$ stable operations $\CK_{\widehat\Z}^*\to \CK_{\widehat\Z}^*$ is canonically isomorphic to the group
\[
S:=\cap_{n}\Im(\Phi^n) \subset \widehat\Z[[x]].
\]
We identify this group in Section \ref{stableoper}. In particular we prove that $S$ is the closure in the $x$-adic topology
of $\widehat{\zz}[[x]]$ of the set of all  (finite) $\widehat{\zz}$-linear combinations of the Adams power series
$A_r$ for $r\in\widehat{\zz}^{\times}$. The
$\widehat{\zz}$-module $S$ and its integral version $S_0$ appear
to be of an uncountable rank. We describe a topological basis
for them.

We call a multiplicative operation $G$ {\it stable} if the constant sequence
$(G,G,G,\ldots)$ is stable. We prove that stable multiplicative operations
$\op{\CK}_{\wzz}^*\row\op{\CK}_{\wzz}^*$ are exactly operations
$\Psi^c_1$, for
$c\in\widehat{\zz}^{\times}$. Thus, we obtain:

\begin{itheorem}
Homogeneous degree $0$
stable additive operations on $\CK^*_{\widehat{\zz}}$ are
topologically generated by the stable multiplicative operations
on it.
\end{itheorem}

Similarly, stable multiplicative operations on
$\op{CK}^*$ are $\Psi^{\pm 1}_1$. This time though, they don't
generate the group of stable additive operations which is of
uncountable rank.

Recall that additive operations in (graded) $K$-theory were determined in \cite[\S 6.1]{vishik19}.
In the present paper we determine stable and multiplicative
operations in $K_{gr}$. We describe a basis of
the group of stable $K_{gr}$-operations and relate it to the
basis of stable $\op{CK}$-operations.
The ring of stable operations is dual to the Hopf algebra of co-operations defined over $\Z$ and therefore
has a structure of (topological) Hopf algebra. The Hopf algebra of co-operations coincides with $K_0(K)$ in topology and has been
studied in \cite{AC77}, \cite{AHS71}, \cite{CCW01}, \cite{Johnson84} and \cite{SW10}. The case of $\op{CK}$
was investigated, in particular, in \cite{Kane81}.

The main tool used in our proofs is the general result of the second author
\cite[Theorem 6.2]{vishik19} that asserts, when applied to the connective $K$-theory, that an operation $G\in \Op_R^{n,m}$
for $n\geq 1$ is given by a sequence of symmetric power series $G_l\in R[[x_1,\ldots, x_l]]$
for all $l\geq n$ satisfying certain conditions. In particular, $G_l$ divisible by $x_1\cdot\ldots\cdot x_{l}$
and $-G_{l+1}=\partial(G_l)$, the
\emph{partial derivative} of $G_l$ (see Definition \ref{parderiv}) for all $l\geq n$, i.e., all power series $G_l$ are
determined by $G_n$. We show that if $R$ is torsion free, then $G_n$ can be integrated over
$K=R\tens \Q$: there is a unique power series $H\in xK[[x]]$ such that $G_n=\partial^{n-1}(H)$. Thus, the operation $G$ is determined
by a power series $H$ in one variable over $K$ such that $\partial^{n-1}(H)\in R[[x_1,\ldots, x_n]]$.

The article is organized as follows.
In Section 2 we prove general results which will permit
us to integrate the multivariate symmetric power series
and reduce the classification of operations to the description
of power series in one variable with certain integrality
properties. These properties are then studied
and the respective power series are classified in Section 3.
In Section 4 we apply the obtained results in combination
with \cite[Theorem 6.2]{vishik19} to produce a description of
additive operations in $\op{CK}$ with integral and
$\widehat{\zz}$-coefficients. We describe the ring structure
on the set of homogeneous operations. The description of
operations in $K_{gr}$ comes as an easy by-product.
In the latter case, we also describe
the dual bi-algebra of co-operations.
Multiplicative operations in $\op{CK}$ and $K_{gr}$ are
studied in Section 5.
Finally, Section 6 is devoted to the computation of stable
operations.

\section{Symmetric power series}

\subsection{Partial derivatives}
Let $F(x,y)$ be a (commutative) formal group law over a commutative ring $R$. We write $x\!*\!y:=F(x,y)$.

Let $G(x_1,\ldots,x_n)\in R[[x_1,\ldots, x_n]]$ be a power series in $n\geq 1$ variables.

\begin{definition}\label{parderiv}
The \emph{partial derivative} of $G$
(with respect to $F$) is the power series
\begin{align*}
(\partial G)(x_1,x_2,\ldots,x_{n+1})&=G(x_1\!*\! x_2,x_3,\ldots,x_{n+1})-G(x_1,x_3,\ldots,x_{n+1}) \\
&-G(x_2,x_3,\ldots,x_{n+1})+G(0,x_3,\ldots,x_{n+1})\in R[[x_1,\ldots, x_{n+1}]].
\end{align*}
\end{definition}

Note that the partial derivative is always taken with respect to
the first variable (in this case $x_1$) in the list of variables.
Write $\partial^m$ for the iterated partial derivative. We also set $(\partial^0 G)(x_1,\ldots,x_n)=G(x_1,\ldots,x_n)-G(0,x_2,\ldots,x_n)$.

For a subset $I\subset [1,m+1]:=\{1,\ldots, m+1\}$ write $x_I$ for the $*$-sum of
all $x_i$ with $i\in I$. In particular, $x_{\emptyset}=0$. Then
\begin{equation}
\label{m-th-deriv}
(\partial^m G)(x_1,\ldots x_{m+n})=
\Sum (-1)^{|I|} G(x_I,x_{m+2},\ldots, x_{m+n})\in R[[x_1,x_2,\dots, x_{m+n}]],
\end{equation}
where the sum is taken over all $2^{m+1}$ subsets $I\subset [1,m+1]$.
In particular, $\partial^m G$ is symmetric with respect to the first $m+1$ variables.

\begin{observation}\label{observe}
If $G\in R[[x_1,\ldots, x_n]]$ is such that $\partial G$ is a symmetric power series, then
$\partial^m G$ is symmetric for all $m\geq 1$.
\end{observation}
Indeed, since $\partial G$ is symmetric, $\partial^m G=\partial^{m-1}(\partial G)$ is symmetric with respect to the last $n$ variables.
But $\partial^m G$ is symmetric with respect to the first $m+1$ variables, hence it is symmetric.

\begin{notation}
For any commutative $\Q$-algebra $K$ write
\[
\lg_1(x):=\log(1-x)=-\Sum_{i\geq 1}\frac{x^i}{i}\in K[[x]]
\]
and for any $n\geq 0$,
\[
\lg_n(x):=\frac{1}{n!}\ \big(\lg_1(x)\big)^n\in K[[x]].
\]
In particular, $\lg_0(x)=1$.
\end{notation}

For the rest of this section $*$ denotes the multiplicative formal group law, i.e., $x\!*\! y=x+y-xy$.

The power series $\lg_1(x)$
belongs to the kernel of $\partial$. Moreover, we have the following statement.

\begin{proposition}\label{kernel}
For any commutative $\Q$-algebra $K$ and any $n>0$, the kernel of $\partial^{n-1}:K[[x]]\to K[[x_1,\ldots,x_{n}]]$ is equal to
\[
\Sum_{0\leq r< n}K\cdot \lg_r(x).
\]
\end{proposition}

\begin{proof}
We change the variables: $y_i=\lg_1(x_i)=\log(1-x_i)$, where $x_1=x$. The multiplicative group law $*$ translates to the additive one. In the new
variables the partial derivative is homogeneous and lowers the degree in $y_1$ by $1$. Therefore, the kernel of $\partial^n$
is spanned by $1, y_1,\ldots, y_1^{n-1}$.
\end{proof}

The following formula is very useful.

\begin{proposition}\label{aformula}
Let $K$ be a $\Q$-algebra, $G\in K[[x]]$ and $n$ a positive integer. Then
\[
(\partial^n G)(x_1,x_2,\dots x_{n+1})=
\Sum_{k=1}^{\infty}\frac{1}{k!}\ \partial^{n-1}\Big((1-x)^k  \diff[k]{G}{x}\Big)(x_1,x_2,\dots, x_n)\cdot x_{n+1}^k.
\]
\end{proposition}

\begin{proof}
Note that both sides don't contain monomials
$\overline{x}^\alpha:=x_1^{\alpha_1}x_2^{\alpha_2}\cdots x_{n+1}^{\alpha_{n+1}}$ if at least one $\alpha_i$ is zero.
We prove that for every multi-index $\alpha$ with $\alpha_i>0$ for all $i$, the $\overline{x}^\alpha$-
coefficients of both sides are equal. Set $k=\alpha_{n+1}$.

By (\ref{m-th-deriv}), the $\overline{x}^\alpha$-coefficient of the left hand side is the same
as the $\overline{x}^\alpha$-coefficient of \newline $G(x_1\! *\! x_2\! *\cdots \! *\!  x_{n+1})$.
To determine this coefficient,
we differentiate (in the standard way) $k$ times the series $G(x_1\! *\! x_2\! *\! \cdots \! *\!  x_{n+1})$ by $x_{n+1}$,
plug in $x_{n+1}=0$ and divide by $k!$.
Since our formal group law is multiplicative, we have
$1-x*y=(1-x)(1-y)$ and so,
\[
\diff{}{x_{n+1}}(x_1\! *\! x_2\! *\! \cdots \! *\!  x_{n+1})=(1-x_1)(1-x_2)\cdots (1-x_n).
\]
It follows that the $\underline{x}^\alpha$-coefficient in the left hand side is equal to
the $x_1^{\alpha_1}x_2^{\alpha_2}\cdots x_{n}^{\alpha_{n}}$-coefficient of
\[
\frac{1}{k!}(1-x_1)^k(1-x_2)^k\cdots (1-x_n)^k \diff[k]{G}{x}(x_1\! *\! x_2\! *\! \cdots \! *\!  x_{n}).
\]

On the other hand, note that the $\underline{x}^\alpha$-coefficient of the right hand side is equal to
the $x_1^{\alpha_1}x_2^{\alpha_2}\cdots x_{n}^{\alpha_{n}}$-coefficient of
$\frac{1}{k!}\partial^{n-1}\Big((1-x)^k \diff[k]{G}{x}\Big)(x_1,x_2,\dots, x_n)$. This is the same as the
$x_1^{\alpha_1}x_2^{\alpha_2}\cdots x_{n}^{\alpha_{n}}$-coefficient of
\[
\frac{1}{k!}(1-x_1*x_2*\cdots * x_{n})^k G^{(k)}(x_1*x_2*\cdots * x_{n})=\frac{1}{k!}(1-x_1)^k(1-x_2)^k\cdots (1-x_n)^k \diff[k]{G}{x}(x_1*x_2*\cdots * x_{n}).
\]
\end{proof}

For a nonzero power series $H \in R[[x_1,\ldots,x_n]]$ denote by $v(H)$ the smallest degree of monomials in $H$.
Set also $v(0)=\infty$.

\begin{observation}\label{valuation}
Suppose that a commutative ring $R$ is torsion free. A direct calculation shows that for positive integers $n$ and $m$,
we have $v(\partial^{n-1}(x^m))=m$ if $m\geq n$. It follows that
$v(\partial^{n-1}(G))=v(G)$ for every $G\in R[[x]]$ such that $v(G)\geq n$.
\end{observation}

\subsection{Integration of symmetric power series}

\begin{definition}
A power series $G\in R[[x_1,\ldots, x_n]]$ is called \emph{double-symmetric}
if $G$ itself and $\partial G$ are both symmetric.
\end{definition}

In the following proposition we prove that double-symmetric power series can be
symmetrically integrated over any commutative $\Q$-algebra.

\begin{proposition} \label{allequivalent}
Let $K$ be a commutative $\Q$-algebra and $G\in K[[x_1,\ldots, x_n]]$, $n\geq 2$,
be a symmetric power series divisible by $x_1\cdot\ldots\cdot x_{n}$. The following are equivalent:
\begin{enumerate}
   \item $G$ is double-symmetric;
  \item All derivatives $\partial^m(G)$, $m\geq 0$, are symmetric power series;
  \item There is a power series $L\in K[[x]]$ such that $G=\partial^{n-1}(L)$;
    \item There is $H\in K[[x_1,\ldots, x_{n-1}]]$ such that $\partial(H)=G$;
  \item There is a unique symmetric $H\in K[[x_1,\ldots, x_{n-1}]]$, divisible by $x_1\cdot\ldots\cdot x_{n-1}$, with zero
  coefficient at $x_1\cdot\ldots\cdot x_{n-1}$ and such that $\partial(H)=G$.
  \end{enumerate}
\end{proposition}

\begin{proof}
Note that $(1)\Leftrightarrow (2)$ by Observation \ref{observe}.
We will prove the equivalence of all statements by induction on $n$.
The implication $(3)\Rightarrow (2)$ is clear, $(2)\Rightarrow (1)$ and $(3)\Rightarrow (4)$ are trivial.

$(5)\Rightarrow (3)$ follows by induction applied to $H$.

$(1)$ or $(4) \Rightarrow (5)$ Over a commutative $\Q$-algebra every formal group law is isomorphic to
the additive one. So we may assume that
the group law is additive, i.e., the derivative is defined by
\[
(\partial G)(x,y,\bar t)=G(x+y,\bar t)-G(x,\bar t)-G(y,\bar t)+G(0,\bar t).
\]

We first prove uniqueness. Indeed if $\partial H=0$,
then $H$ is linear in $x_1$, and since $H$ is symmetric and divisible by $x_1\cdot\ldots\cdot x_{n-1}$,
we must have $H=0$.

\smallskip

Case $n=2$: The implication $(4) \Rightarrow (5)$ is obvious. We prove $(1) \Rightarrow (5)$. We may assume that $G$ is a homogeneous polynomial of degree $d>1$.
The symmetry of the derivative
of $G(x,y)$ results in the following cocycle condition:
\[
G(x+y,z)+G(x,y)=G(x+z,y)+G(x,z).
\]
In particular, we have the following equalities:
\begin{align*}
G(x+y,x+y)+G(x,y)&=G(2x+y,y)+G(x,x+y),\\
G(2x+y,y)+G(2x,y)&=G(2x,2y)+G(y,y),\\
G(x,x+y)+G(x,y)&=G(2x,y)+G(x,x).
\end{align*}
It follows that
\begin{align*}
\partial(G(x,x))(x,y)&=G(x+y,x+y)-G(x,x)-G(y,y)\\
&=G(2x,2y)-2G(x,y)\\
&=(2^d-2) (G(x,y)),
\end{align*}
hence $G(x,y)=\partial(H)$, where $H(x)=G(x,x)/(2^d-2)$.

\smallskip

Case $n=3$: Write $G(x,y,z)=\Sum_{i\geq 1} G_i(x,y)z^i$. By the very definition, if $G$ satisfies $(1)$, respectively $(4)$, then all $G_i(x,y)$ also
satisfy $(1)$, respectively $(4)$. By induction, they satisfy $(5)$.
Integrating each $G_i(x,y)$, we get a power series
$H=\sum_{i,j\geq 1}a_{i,j}x^iy^j$ in two variables such that $\partial H=G$.

Note that
we can change $H$ by any series $\sum_i c_ixy^i$ without changing
$\partial H$. This way, we can make $H=\sum_{i,j\geq 1}a_{i,j}x^iy^j$ with $a_{i,1}=a_{1,i}$ and $a_{1,1}=0$.
We claim that $H$ is symmetric. Indeed, from the symmetry of $\partial H$, we have:
$$
\mybinom[0.8]{i+k}{i} a_{i+k,j}=\mybinom[0.8]{j+k}{j} a_{j+k,i},
$$
for any $i,j,k\geq 1$. This implies that
\[
{\textstyle\frac{1}{i+l}}\mybinom[0.8]{i+l}{i}a_{i+l-1,1}=a_{l,i},
\]
and so, $a_{i,l}=a_{l,i}$, for any $i,l\geq 2$.
This shows that $H$ is symmetric.
Observe that such symmetric integration is unique provided $a_{1,1}=0$.

\smallskip

Case $n>3$: Write $G=\Sum_{i\geq 1}G_i\cdot x_n^i$ with $G_i\in K[[x_1,\ldots, x_{n-1}]]$.
Again, by the very definition, the slices $G_i$ of $G$ are double-symmetric. By the inductive assumption, these can be uniquely
integrated to symmetric power series $H_i\in K[[x_1,...,x_{n-2}]]$ as in $(5)$.
Putting these power series together, we obtain
\[
H=\Sum_{i\geq 1}H_i\cdot x_{n-1}^i\in K[[x_1,...,x_{n-1}]]
\]
such that $\partial H=G$. Write
\[
H=\Sum_{i_1,\ldots,i_{n-1}} a_{i_1,\ldots,i_{n-1}}x_1^{i_1}\ldots x_{n-1}^{i_{n-1}}.
\]
Modifying $H$ by $x_1 \ldots x_{n-1}L(x_{n-1})$ for an appropriate power series $L$, we may assume that
$a_{i,1,\ldots,1}=a_{1,1,\ldots,i}$ for all $i$.

We claim that $H$ is symmetric. The $x_1^{i_1}\ldots x_{n}^{i_{n}}$-coefficient of $G=\partial H$
is equal to $\mybinom[0.7]{i_1+i_2}{i_1}a_{i_1+i_2,i_3,...,i_n}$. Therefore, since $G$ is symmetric,
$H$ is symmetric with respect
to $x_2,...,x_{n-1}$, if $i_1>1$. Recall than $H$ is also symmetric in $x_1,\ldots, x_{n-2}$. Therefore,
it suffices to show that the coefficient $a_{1, i_2,\ldots,i_{n-1}}$ does not change if we interchange
$i_{n-1}$ with $i_k$ for some $k=2,\ldots, n-2$.

Suppose all indices $i_1,\ldots,i_{n-1}$ but one are equal to $1$. Then the statement follows from the equality
$a_{1,1,\ldots,i}=a_{i,1,\ldots,1}=a_{1,i,\ldots,1}$ for all $i$. Otherwise, at least two indices, say
$i_k=u$ and $i_l=v$ with $k<l$ are greater than $1$.

If $l<n-1$, set $w=i_{n-1}$. We have (here and below we indicate only the indices which
are permuted, hidden indices remain unchanged):
\[
a_{1,u,v,w}=a_{v,u,1,w}=a_{v,w,1,u}=a_{1,w,v,u},
\]
so we interchanged $i_k$ and $i_{n-1}$. If $l=n-1$, we can write
\[
a_{1,u,v}=a_{u,1,v}=a_{u,v,1}=a_{v,u,1}=a_{v,1,u}=a_{1,v,u},
\]
i.e., we again interchanged $i_k$ and $i_{n-1}$.
\end{proof}

\section{The groups ${\cQ}^n_R$}\label{firstfirst}

The formal group law is multiplicative in this section. Let $R$ be a commutative ring and
$K=R\tens_{\Z} \Q$. We assume that $R$ is torsion free (as abelian group), i.e., $R$
can be identified with a subring of $K$.

\begin{definition}\label{qn}
For any integer $n\geq 1$, let us denote by $\cQ_R^n$ the $R$-module of power series $G$ in $xK[[x]]$, for which
$\partial^{n-1}(G)\in R[[x_1,...,x_n]]$. For example, $\cQ_R^1=xR[[x]]$.
We also set $\cQ_R^n=R[[x]]$ if $n\leq 0$.
\end{definition}

Note that $xR[[x]]$ and $\Sum_{0< r< n}K\cdot \lg_r(x)$ are contained in $\cQ_R^n$ in view of Proposition \ref{kernel}.
In Theorem \ref{mainone} below we will see that the quotient of $\cQ_R^n$ by the second of these subspaces  can be identified with the space of additive operations on $\CK^n_R$.

\begin{lemma}\label{divdiv}
Suppose $R$ has no nontrivial $\Z$-divisible elements. Then
\[
xR[[x]]\cap \Big(\Sum_{0< r< n}K\cdot \lg_r(x)\Big)=0.
\]
\end{lemma}

\begin{proof}
Consider the operator $\Phi$ on $K[[x]]$ mapping $R[[x]]$ to itself:
\[
\Phi(F(x)):=(x-1)\cdot\diff{}{x}\left(F(x)\right).
\]
Observe that
$
\Phi(\lg_r(x))=\lg_{r-1}(x).
$
Suppose $\Sum_{0< r<n}q_r\cdot\lg_{r}(x)\in xR[[x]]$, where $q_r\in K$ and let $r$ be
the largest index such that $q_r\neq 0$. Applying $\Phi^{r-1}$ to the sum we see that $q_{r-1}+q_r\lg_{1}(x)\in R[[x]]$.
Let $n\in\nn$ be a natural number such that $nq_{r-1}\in R$ and $nq_{r}\in R$. It follows that $nq_{r}\in iR$ for every
integer $i>0$, i.e., $nq_{r}$ is a nonzero $\Z$-divisible element in $R$, a contradiction.
\end{proof}

\begin{definition}\label{qn2} Let $n$ and $m$ be integers. If $n>0$ denote by ${\cQ}_R^{n,m}$ the submodule of ${\cQ}_R^{n}$ consisting
of all power series $G$ such that $v(\partial^{n-1}G)\geq m$. If $n\leq 0$, set ${\cQ}_R^{n,m}=x^{\max{(0,m)}}\cdot R[[x]]$.
\end{definition}

Theorem \ref{mainone} permits to describe the $R$-module of operations $\Op_R^{n,m}$ in terms of the modules ${\cQ}_R^{n,m}$.

Since $v(\partial^{n-1}G)\geq n$ for every $G\in {\cQ}_R^{n}$ with $n>0$, we have ${\cQ}_R^{n,m}={\cQ}_R^{n,n}={\cQ}_R^{n}$ if $n\geq m$.
Note also that ${\cQ}_R^{1,m}=x^{\max{(1,m)}}\cdot R[[x]]$.

\subsection{The groups ${\cQ}^n_{\widehat \Z}$}
In this section we determine the structure of the modules ${\cQ}^n_{\widehat \Z}$ over the ring $\widehat \Z=\lim(\Z/n\Z)$.
We write $\widehat \Q$ for $\widehat \Z\tens\Q$. Note that $\widehat \Q=\widehat \Z +\Q$ and
$\Z=\widehat \Z \cap \Q$ in $\widehat \Q$.

\begin{lemma}\label{second}
Let $b_1,b_2,\dots, b_m\in \widehat \Z$ be such that $b_i\equiv b_j$ $(mod\ j)$ for every $i$ divisible by $j$.
Then there is $b\in \Z$ such that $b\equiv b_i$ $(mod\ i)$ for all $i=1,\dots,m$.
\end{lemma}
\begin{proof}
Let $p_1,p_2,\dots,p_s$ be all primes at most $m$. For every $k$, let $q_k=p_k^{r_k}$ be the largest power of $p_k$ such that
$q_k\leq m$. By Chinese Remainder Theorem, we can find $b\in\Z$ such that $b\equiv b_{q_{k}}$ $(mod\ q_k)$ for all $k$.
We claim that $b$ works. Take any $i\leq m$. We prove that $b\equiv b_i$ $(mod\ i)$.
Write $i$ as the product $i=\prod q'_k$, where $q'_k$ is a power of $p_k$.
Clearly, $q'_k$ divides $q_k$. We have
\begin{align*}
    b_{q'_{k}}  \equiv&\ b_i\quad (mod \ \ q'_{k})\quad\text{by assumption},\\
    b_{q_{k}}\equiv&\ b_{q'_{k}}\quad (mod \ \ q'_{k})\quad\text{by assumption},\\
    b\equiv&\ b_{q_{k}}\quad (mod \ \ q_{k})\quad\text{by construction}.
\end{align*}
It follows that $b\equiv b_i$ $(mod \ q'_{k})$ for all $k$, hence $b\equiv b_i$ $(mod \ i)$.
\end{proof}

Let $G(x)=\sum_{i=1}^{\infty} a_ix^i$ with $a_i\in\widehat\Q$.
\begin{lemma}\label{first}
For positive integers $j\leq s$, the $x^jy^s$-coefficient of $\partial G$ is equal to
\[
 \Sum_{i=0}^{j}(-1)^{j-i}\mybinom[0.8]{s+i}{s}\mybinom[0.8]{s}{j-i}a_{s+i}.
 \]
\end{lemma}
\begin{proof}
We have
\[
\frac{1}{s!}\diff[s]{G}{x}=\Sum_{i=0}^{\infty}\mybinom[0.8]{s+i}{s}a_{s+i}x^i.
\]
The statement follows from Proposition \ref{aformula}.
\end{proof}

Set $b_i=ia_i$ for all $i\geq 1$.

\begin{corollary}\label{inte}
If $\partial G\in \widehat\Z[[x,y]]$ then $b_i-b_1\in \widehat\Z$ for all $i\geq 1$.
In particular, if $a_1\in\widehat\Z$, then all $b_i$ are in $\widehat\Z$.
\end{corollary}
\begin{proof}
The $xy^j
$-coefficient of $\partial G$ is equal to $b_{j+1}-b_j$.
\end{proof}

\begin{proposition}\label{main}
Let $G\in {\cQ}^2_{\widehat \Z}$ and let $n>1$ be an
integer such that $a_i\in\widehat\Z$ for all $i<n$. Let $p^t<n$ be power of a prime integer $p$
such that $p^t$ divides $n$. Then $p^t$ divides $b_n$.
\end{proposition}

\begin{proof}
Take $j=p^t$ and $s=n-p^t\geq p^t$.
By Lemma \ref{first}, the $x^jy^s$-coefficient of $\partial G$ is equal to
\[
 \Sum_{i=0}^{j}(-1)^{j-i}\mybinom[0.8]{s+i}{s}\mybinom[0.8]{s}{j-i}a_{s+i}\in\widehat\Z.
\]
 By assumption, all terms in the sum but the last one belong to $\widehat\Z$, hence so does the
 last one: $\binom{n}{p^t}a_{n}\in\widehat\Z$. But $\binom{n}{p^t}a_{n}=\binom{n-1}{p^t-1}b_{n}/p^t$,
 hence $\binom{n-1}{p^t-1}b_{n}$ is divisible by $p^t$. As $\binom{n-1}{p^t-1}$ is prime to $p$,
 the coefficient $b_{n}$ is divisible by $p^t$
 (recall that $\binom{a+b}{a}$ is relatively prime to $p$ if and only if there is no shift of digits in the long addition of $a$ and $b$ written in the $p$-base).
\end{proof}

\begin{proposition}\label{onevar}
We have
\[
{\cQ}_{\widehat\Z}^2=\widehat \Q\cdot\lg_1(x)\oplus x\widehat\Z[[x]].
\]
\end{proposition}

\begin{proof}
Let $G(x)=\sum_{i=1}^{\infty} a_ix^i\in {\cQ}_{\widehat\Z}^2$ and set as before $b_i=ia_i$.
Adding $a_1\lg_1(x)$ to $G(x)$ we may assume that $a_1=0$.
By Corollary \ref{inte}, $b_i\in{\widehat\Z}$ for all $i$.

We claim that for every positive integer $i<n$ such that $i$ divides $n$ we have
$b_n\equiv b_i$ modulo $i$.
We prove this by induction on $n$. By Lemma \ref{second} applied to $m=n-1$,
there is $b\in\Z$ such that $b\equiv b_i$ modulo $i$ for all $i<n$. Subtracting $b\lg_1(x)$
from $G(x)$, we may assume that $b_i$ is divisible by $i$ for all $i<n$, or equivalently, $a_i\in\widehat\Z$
for all $i<n$. We prove that $b_n$ is divisible by $i$, for every $i<n$ dividing $n$.

Case 1: $n=p^k$ is a power of a prime $p$. Then $i=p^t$ is a smaller power of $p$.
By Proposition \ref{main}, $i$ divides $b_n$.

Case 2: $n$  is not power of a prime. Write $n$ as a product of powers of distinct primes:
$n=q_1 q_2 \cdots q_s$. By Proposition \ref{main}, $q_k$ divides $b_n$ for every $k$, hence $n$ divides $b_n$.
In particular, $i$ divides $b_n$. The claim is proved.

Let $b\in\widehat\Z$ be such that $b \equiv b_n$ $(mod \ n)$ for all $n$. We have
\[
G=b\lg_1(x)+\Sum_{n\geq 1} \frac{b_n-b}{n}x^n\in \widehat \Z\cdot\lg_1(x)+ x\widehat\Z[[x]].\qedhere
\]
\end{proof}

\begin{corollary}\label{intee}
Let $G(x)=ax+\ldots\in {\cQ}^2_{\widehat \Z}$ be a power series with $a\in\widehat \Z$. Then $G(x)\in \widehat\Z\cdot \lg_1(x)+x\widehat\Z[[x]]$.
\end{corollary}

In analogy with partial derivative with respect to the first variable - Definition \ref{parderiv}, we may define the partial derivative
with respect to any other variable.
In the next statement we will use such partial derivatives for $H(x,y)$. In particular,
\[
 (\partial_yH)(x,y,z)=H(x,y*z)-H(x,y)-H(x,z)+H(x,0).
\]

\begin{lemma}
\label{intH}
Let $H(x,y)=\sum_{i,j\geq 1}a_{i,j}x^iy^j\in\widehat\Q[[x,y]]$ be such power series that both $\partial$-partial derivatives of $H$ have  coefficients in $\widehat\Z$ and
$a_{i,1}$ as well as $a_{1,i}$ are in $\widehat\Z$, for all $i$. Then $H(x,y)\in\widehat\Z[[x,y]]$.
\end{lemma}

\begin{proof}
Consider some $j$-th row of $H$: $y^j\cdot\sum_{i\geq 1}a_{i,j}x^i$. We know that
$\sum_{i\geq 1}a_{i,j}x^i\in\cQ^2_{\widehat\Z}$.
By Corollary \ref{intee}, $\sum_{i\geq 1}a_{i,j}x^i$ is equal to $c_j\cdot\lg_1(x)$ modulo $x\widehat{\Z}[[x]]$
for some $c_j\in\widehat{\Z}$. Hence
$\displaystyle\frac{c_j}{i}\equiv a_{i,j}\,\,(mod\ \widehat{\Z})$ for all $i$. Applying the same considerations to the $i$-th column $x^i\cdot\sum_{j\geq 1} a_{i,j}y^j$, we obtain:
$$
\frac{c_j}{i}\equiv\frac{d_i}{j}\hspace{2mm}(mod\ \widehat{\Z}),
$$
for certain $d_i\in\widehat{\Z}$.
Let us show that all $c_i$'s (and $d_j$'s) are zeros. Indeed, we have:
$$
jc_j\equiv id_i\hspace{2mm}(mod\,ij).
$$
Hence, $jc_j$ is divisible by $i$, for any $i$ and, hence $c_j=0$. This implies that
$a_{i,j}\in \widehat{\Z}$ for any $i,j$.
\end{proof}

\begin{lemma}
\label{nintH}
Suppose, $H(x_1,...,x_n)=\sum_{i_1,...,i_n>0}a_{i_1,...,i_n}x_1^{i_1}\cdot\ldots\cdot x_n^{i_n}\in\widehat\Q[[x_1,\ldots,x_n]]$
is such a power series that all $\partial$-partial derivatives of $H$ with respect to all variables have coefficients in
$\widehat\Z$ and $a_{i_1,...,i_n}\in\widehat\Z$ as long as all $i_j$'s but one are equal to $1$.
Then $H$ has coefficients in $\widehat\Z$.
\end{lemma}

\begin{proof}
Induction on $n$. For $n=1$ there is nothing to prove. For $n=2$ this is Lemma \ref{intH}. We can assume that $n\geq 3$.
Suppose we know the statement for $r<n$.
Let $L\subset [1,n]$ be some subset. Consider
the sum of monomials of $H$ with $i_j=1$, for every $j\in L$. Plugging $x_j=1$ for all $j\in L$, we obtain the power series
in variables $x_j, j\not\in L$ which we will call the
$L$-cell $H_L$ of $H$.
Similarly, considering the sum of the monomials of $H$
with the given $i_1$ and plugging $x_1=1$ into it, we
get the power series $H_{i_1}(x_2,\ldots,x_n)$ which we call the hyper-slice of $H$.
Note, that all the cells of $H$
satisfy the conditions of the Lemma. By our assumption, these have all coefficients in $\widehat\Z$.
That is, $a_{i_1,...,i_n}\in\widehat\Z$ provided,
at least, one of $i_j$'s is $1$. The hyper-slice $H_{i_1}$ satisfies the conditions of the
Lemma too (note that $n\geq 3$). Thus, $H_{i_1}$ has coefficients in $\widehat\Z$ and so does $H$.
\end{proof}

The following theorem is a generalization of Proposition \ref{onevar}.

\begin{theorem}\label{description}
For every $n\geq 1$,
\[
{\cQ}_{\widehat\Z}^{n}=\Coprod_{0<r<n}\widehat \Q\cdot\lg_{r}(x)\oplus x\widehat\Z[[x]].
\]
\end{theorem}

\begin{proof}
The statement is clear if $n\leq 0$. Now assume that $n\geq 1$.
It follows from Lemma \ref{divdiv} that $\Coprod_{0<r<n}\widehat \Q\cdot\lg_{r}(x)\cap \widehat\Z[[x]]=0$.

We prove the rest by induction on $n$.
For $n=1$ this is so by definition and for $n=2$ this is given by Proposition \ref{onevar}.

\smallskip

\noindent $n\Rightarrow n+1$:
Let $G\in{\cQ}_{\widehat\Z}^{n+1}$. Consider the power series $H(x_1,...,x_n)=\partial^{n-1} (G)$.
Let
\[
H(x_1,...,x_n)=\Sum_{i_1,...,i_n\geq 1}a_{i_1,...,i_n}x_1^{i_1}\cdot...\cdot x_n^{i_n}.
\]

Note that the degree-wise smallest term of $\partial^{n-1}(\lg_{n}(x))$ is $(-1)^nx_1\cdot...\cdot x_n$.
By subtracting an appropriate $\widehat\Q$-multiple of $\lg_{n}(x)$
from $G$, we may assume that $a_{1,...,1}=0$.

As $\partial(H)$ has coefficients in $\widehat\Z$, the ``ray'' $\sum_{i\geq 1} a_{i,1,...,1}x_1^i$
is a power series with terms of degree $\geq 2$ whose $\partial$-derivative is integral.
By Corollary \ref{intee}, up to a power series in $\widehat\Z[[x_1]]$, it is equal to
$c\cdot\lg_1(x_1)$, for some $c\in\hat{\Z}$.

Since
\[
\partial^{n-1}(\lg_n)(x_1,...,x_n)=\lg_1(x_1)\cdot\ldots\cdot\lg_1(x_n),
\]
subtracting from $G(x)$ an appropriate multiple of $\lg_{n}(x)$, we may assume that the coefficients
$a_{i,1,...,1}$ are in $\widehat\Z$, for all $i\geq 1$. Since $H$ is symmetric, by Lemma \ref{nintH},
all coefficients of the power series $H$ are in $\widehat\Z$.
By the induction hypothesis, $G(x)\in{\cQ}^{n}=\Coprod_{0<r<n}\widehat \Q\cdot \lg_{r}(x)+x\widehat\Z[[x]]$.
\end{proof}

\subsection{The groups ${\cQ}^n$}
Write ${\cQ}^n$ for ${\cQ}^n_{\Z}\subset {\cQ}^n_{\widehat\Z}$.

We define a homomorphism
\[
\rho_n:\cQ^{n}\to \widehat \Q^{n-1}
\]
for $n\geq 1$ as the composition (see Theorem \ref{description})
\[
\cQ^{n}\hookrightarrow {\cQ}^n_{\widehat\Z}= \Coprod_{0<r<n}\widehat \Q\cdot\lg_{r}(x)\oplus x\widehat\Z[[x]]
\xra{\operatorname{proj}} \Coprod_{0<r<n}\widehat \Q\cdot\lg_{r}(x)\simeq \widehat \Q^{n-1}.
\]
We will show that the map $\rho_n$ is surjective.

Consider the power series
$$
\widetilde{\lg}_r(x)= (-1)^r\Sum_{0<i_1<...<i_r}\frac{x^{i_r}}{i_1\cdot ...\cdot i_r}\equiv(-1)^r\frac{x^r}{r!}\ \ \ (mod\ x^{r+1}).
$$

For a sequence of $a=(a_i)_{i\geq 1}$ in $\widehat \Z$ let us denote by
$a\cdot\widetilde{\lg}_r(x)\in\widehat\Q[[x]]$ the power series
$$
(-1)^r\Sum_{0<i_1<...<i_r}\frac{a_{i_1}\cdot x^{i_r}}{i_1\cdot ...\cdot i_r}\equiv (-1)^r\frac{a_1}{r!}x^r\ \ \ (mod\ x^{r+1}).
$$
If all $a_i\in\Z$, we have $a\cdot\widetilde{\lg}_r(x)\in\Q[[x]]$.

\begin{lemma}\label{forrmulla}
\label{logrder}
For every sequence $a$, we have
 $$
 (x-1)\cdot\diff{}{x}\left(a\cdot\widetilde{\lg}_r(x)\right)=a\cdot\widetilde{\lg}_{r-1}(x).
 $$
\end{lemma}

\begin{proof}
Write $(-1)^r a\cdot\widetilde{\lg}_r(x)=\sum b_i x^i$ and $(-1)^{r-1} a\cdot\widetilde{\lg}_{r-1}(x)=\Sum c_i  x^i$. We need to
prove that $(m+1)b_{m+1}-mb_m=c_m$ for every $m$. We have
\[
(m+1)b_{m+1}=\Sum_{0<i_1<...<i_{r-1}<m+1}\frac{a_{i_1}}{i_1\cdot ...\cdot i_{r-1}}.
\]
The sum of the terms with $i_{r-1}<m$ is equal to $mb_m$. The sum of the terms  with $i_{r-1}=m$ coincides with $c_m$.
\end{proof}

In particular, $\Phi(\widetilde{\lg}_r(x))=\widetilde{\lg}_{r-1}(x)$. Note that we also have
$\Phi(\lg_r(x))=\lg_{r-1}(x)$ and series
$\widetilde{\lg}_r(x)$ and $\lg_r(x)$ have no constant terms for $r\geq 1$. Since the kernel of $\Phi$ consists of constants only and $\widetilde{\lg}_1(x)=\lg_1(x)$, by definition, it follows by induction on $r$ that $\widetilde{\lg}_r(x)=\lg_r(x)$, for
all $r$. In particular, we can define the product
$a\cdot\lg_r(x)$ as above.




\begin{lemma}\label{N33}
For every $c\in\widehat \Z$ and every integer $r>0$ there is a sequence $\tilde c=(c_i)_{i\geq 1}$
of integers $c_i\in\Z$ such that
$c_i\equiv c\ (mod\ i)$ for all $i$ and
\[
(c- \tilde c)\cdot \lg_{i}(x)\in \widehat \Z[[x]]
\]
for all $i=1,\ldots, r$, where $c- \tilde c$ is the sequence $(c- c_i)_{i\geq 1}$.
\end{lemma}

\begin{proof}
Take any collection $\tilde c=(c_i)_{i\geq 1}$ of integers. Note that for every $i\geq 1$ and $k=1,\ldots, r$,
the $x^{i+k-1}$-coefficient of $\tilde c\cdot \lg_{k}(x)$
is a linear combination of $c_1,\ldots, c_{i}$ with rational coefficients where
the $c_{i}$-coefficient is equal to $(-1)^{k}/(i(i+1)\ldots (i+k-1))$.

We will modify $c_1, c_2,\ldots$ inductively to make all coefficients of the power series
\[
G_k=(c- \tilde c)\cdot \lg_{k}(x)
\]
integral for all $k=1,\ldots, r$. Let $c_1$ be an integer congruent to $c$ modulo $r!$, so the $x^k$-coefficient of $G_k$ is integer
for every $k=1,\ldots, r$. Suppose we have modified $c_1,\ldots c_n$
so that the $x^j$-coefficient of $G_k$ is integral
for all $k=1,\ldots, r$ and $j\leq n+k-1$.

By induction on $k=1,\ldots, r$, we will modify $c_{n+1}$ to make integral the $x^{n+k}$-coefficient of $G_k$.
Note that the integral $x^j$-coefficients of $G_k$ for $j\leq n+k-1$ will not change. If $k=1$ we don't modify $c_{n+1}$:
the power series $G_1$ is already integral.

$k\Rightarrow k+1$: By Lemma \ref{logrder},
\[
(x-1)\cdot\diff{G_{k+1}}{x}=G_{k}.
\]
Hence, if $G_{k}=\Sum_{i\geq k}b_ix^i$ and $G_{k+1}=\Sum_{i\geq k+1}a_i x^i$, then
\[
a_{n+l+1}=\textstyle{\frac{-1}{n+l+1}}(b_{k}+\ldots +b_{n+l})
\]
for all $l$.

By induction, $b_k,\ldots, b_{n+k}$ are integral. Recall that these are linear combinations of the $c'_i$'s,
where $c'_i=c-c_i$ and $c'_{n+1}$
appears only in $b_{n+k}$.
We modify $c_{n+1}$ by adding to $c_{n+1}$ the integer $t(n+1)(n+2)\ldots (n+k)$ with some $t\in \Z$. Note that $b_k,\ldots, b_{n+k-1}$
remain unchanged and $b_{n+k}$ changes to $b_{n+k}+t$, so it stays integral. Choose $t$ to make $a_{n+k+1}$ integral.

Note that $c'_{n+1}$ comes with coefficient
$(-1)^l/((n+1)\ldots (n+l))$ in the $x^{n+l}$-coefficient of $G_l$. Since $(n+1)\ldots (n+l)$ divides $(n+1)\ldots (n+k)$ when $l\leq k$,
the $x^{n+l}$-coefficient of $G_l$ remains integral for $l\leq k$.
\end{proof}

Now we prove that the map $\rho_n:\cQ^n\to \widehat \Q^{n-1}$ is surjective. Since $q\cdot \lg_r\in \cQ^n$ for all
$q\in\Q$ and $r=1,\ldots, n-1$, we have $\Q^{n-1}\subset \Im(\rho_n)$. It suffices to show that $\widehat \Z^{n-1}\subset \Im(\rho_n)$.
Choose $c_r\in \widehat \Z$ for $r=1,\ldots, n-1$. By Lemma \ref{N33}, there are sequences of integers
$\tilde c_r$ such that $(c_r - \tilde c_r)\cdot \lg_{r}(x) \in \widehat \Z[[x]]$.

As
\[
\tilde c_r\cdot \lg_{r}(x) =    c_r\cdot \lg_{r}(x)- (c_r- \tilde c_r)\cdot \lg_{r}(x),
\]
we have $\rho_n\(\Sum_{0<r<n}\tilde c_r\cdot \lg_{r}(x)\)=(c_r)_{r=1,\ldots, n-1}$ proving that $\rho_n$ is surjective.

Note that the kernel of $\rho_n$ is equal to $x\widehat\Z[[x]]\cap \Q[[x]]=x\Z[[x]]$. Thus, we have an exact sequence
\begin{equation}\label{sesimp}
0\to x\Z[[x]]\to \cQ^n\xra{\rho_n}\widehat \Q^{n-1}\to 0.
\end{equation}

We have proved that if $n\geq 1$, the group $\cQ^n$ is generated by $x\Z[[x]]$ and the power series $(c- \tilde c)\cdot \lg_{r}(x)$
as in Lemma \ref{N33}, where $c\in \widehat\Z$ and $r=1,\ldots, n-1$.

\smallskip

The power series in $\cQ^n$ can be approximated by polynomials as follows:

\begin{lemma}\label{approxim}
For every $m>0$ and $n$, we have
\[
\cQ^n\subset \Z[x]_{\leq m-1}+ \Sum_{0<r<n}\Q\cdot\lg_{r}(x)+ x^m\Q[[x]],
\]
where $\Z[x]_{\leq m-1}$ is the group of integral polynomials of degree at most $m-1$.
\end{lemma}

\begin{proof}
We may assume that $n>1$. In view of (\ref{sesimp}), the group $\cQ^n$ modulo
\[
x\Z[[x]]+\Sum_{0<r<n}\Q\cdot\lg_{r}(x)
\]
is generated by power series
of the form $\tilde c\cdot \lg_r(x)$, where $r=1,\ldots, n-1$ and $\tilde c$ is the collection of integers
such that $c_i\equiv c\ (mod\ i)$ for all $i$ for an element $c\in\widehat\Z$ as in Lemma \ref{N33}.

Let $d$ be an integer congruent to $c$ modulo the least common multiple of the denominators of the
$x^i$-coefficients of $\lg_r(x)$ for all $i=1,\ldots, m-1$. Then the $x^m$-truncation $F$ of $(\tilde c-d)\cdot \lg_r(x)$
is contained in $\Z[x]_{\leq m-1}$ and $\tilde c\cdot \lg_r(x)$ is congruent to $F$ modulo $\Z\cdot\lg_{r}(x)+x^m\Q[[x]]$.
\end{proof}

\section{Operations}

Let $k$ be a field of characteristic $0$ and write $\smk$ for the category of smooth quasi-projective varieties over $k$. An \emph{oriented
cohomology theory} $A^*$ over $k$ is a functor from $\smk^{op}$ to the category of $\Z$-graded commutative rings equipped with a
push-forward structure and satisfying certain axioms (see \cite[Definition 2.1]{vishik19}). We write
\[
A^*(X)=\Coprod_{n\in\Z} A^n(X)
\]
for a variety $X$ in $\smk$ and let $A^*(k)$ denote the \emph{coefficient ring} $A^*(\Spec k)$.

Let $A^*$ be an oriented cohomology theory. There is a (unique) associated formal group law

\[
F_A(x,y)=\Sum_{i,j\geq 0}a^A_{i,j}x^iy^j=x+y+a_{1,1}\cdot xy+ \text{higher terms}\in A^*(k)[[x,y]]
\]
that computes the first Chern class of the tensor product of two line bundles $L$ and $L'$
(see, for example, \cite[p.3 and Section 3.9]{Panin03}, \cite[Section 2.7]{Panin04}, \cite[\S 1.1]{LM07} or \cite[\S 2.3]{vishik19}):
\[
c_1^A(L\tens L')=F_A(c_1^A(L),c_1^A(L')).
\]

\begin{example}
The \emph{Chow theory} $\CH^*$ takes a smooth variety $X$ to the Chow ring $\CH^*(X)$ of $X$.
We have $\CH^*(k)=\Z$ and $F_{\CH}(x,y)=x+y$ is the \emph{additive} group law.
\end{example}

\begin{example} (see \cite[Example 1.15]{LM07})
The \emph{graded $K$-theory} $K_{gr}^*$ takes $X$ to the Laurent polynomial ring $K_0(X)[t,t\inv]$
(graded by the powers of the \emph{Bott element} $t$
of degree $-1$) over the Grothendieck ring $K_0(X)$ of $X$.
We have $K_{gr}^*(k)=\Z[t,t\inv]$ and $F_{K_{gr}}(x,y)=x+y-txy$ is the \emph{multiplicative} group law.
\end{example}

\begin{example} (see \cite{Cai08} and \cite{DL2014})
The \emph{connective $K$-theory} takes $X$ to the ring $\CK^*(X)$ of $X$.
We have $\CK^*(k)=\Z[t]$ and $F_{CK}(x,y)=x+y-txy$.
\end{example}

All cohomology theories in these examples are of \emph{rational type} (see \cite[\S 4.1]{vishik19} and \cite{LM07}).

If $A^*$ is an oriented cohomology theory and $R$ a commutative ring, the functor
$A_R^*$ defined by $A_R^*(X)=A^*(X)\tens_{\Z} R$
is also an oriented cohomology theory with values in the category of graded $R$-algebras.

\begin{definition}
Let $A^*$ and $B^*$ be two oriented cohomology theories. An \emph{$R$-linear operation}
$G:A_R^*\to B_R^*$ is a morphism between functors $A_R^*$ and $B_R^*$ considered as contravariant functors from
$\smk$ to the category of $R$-modules (cf. \cite[Definition 3.3]{vishik19}). Note that $G$ may not respect the
gradings on $A_R^*$ and $B_R^*$.
\end{definition}

Let $n,m\in\Z$.  A morphism $G:A_R^n\to B_R^m$ between contravariant functors from
$\smk$ to the category of $R$-modules can be viewed as an $R$-linear operation via the obvious composition
$A_R^*\to\!\!\!\!\to A_R^n\to B_R^m\hookrightarrow B_R^*$.
All such operations form
an $R$-module $\Op_R^{n,m}(A^*,B^*)$. The composition of operations yields an $R$-linear pairing
\[
\Op_R^{n,m}(A^*,B^*)\tens_R \Op_R^{m,r}(B^*,C^*)\to \Op_R^{n,r}(A^*,C^*).
\]
In particular, $\Op_R^{n,n}(A^*):=\Op_R^{n,n}(A^*,A^*)$ has a structure of an $R$-algebra.

\begin{example}(see \cite{Cai08} and \cite{DL2014})
Multiplication by $t$ yields an operation $\CK_R^{n+1}\to\CK_R^{n}$ that is an isomorphism if $n< 0$.
There are graded $R$-linear operations
\[
\CK_R^*\to \CH_R^* \quad \text{and}\quad \CK_R^*\to (K^*_{gr})_R.
\]

The sequence
\[
\CK^{n+1}(X)\xra{t}\CK^{n}(X)\to\CH^n(X)\to 0
\]
is exact for every $n$ and $X$.

If $n\geq 0$ the image of the homomorphism $\CK^n(X)\to K_{gr}^n(X)= K_0(X)t^{-n}\simeq K_0(X)$ is generated by
the classes of coherent $\cO_X$-modules with codimension of support at least $n$. If $n\leq 0$ this map is
an isomorphism.
\end{example}

The following fundamental theorem was proved in \cite[Theorem 6.2]{vishik19}.

\begin{theorem}\label{vishik}
Let $A^*$ be a cohomology theory of rational type and $B^*$ be any oriented cohomology theory over $k$.
Let $R$ be a commutative ring.
 Then there is an $R$-isomorphism between the
set $\Op_R^{n,m}(A^*,B^*)$ of $R$-linear operations $G:A_R^n\to B_R^m$ and the set consisting of the following data $\{G_l, l\in\Z_{\geq 0}\}$:
\[
G_l\in \Hom_{R}\big(A^{n-l}(k)\tens R,B^*(k)[[x_1,\ldots, x_l]]_{(m)}\tens R\big)\quad\text{satisfying}
\]
\begin{enumerate}
\item $G_l(\alpha)$ is a symmetric power series for all $l$ and $\alpha\in A^{n-l}(k)\tens R$,
  \item $G_l(\alpha)$ is divisible by $x_1\cdot\ldots\cdot x_l$ for all $l$ and $\alpha$,
    \item $G_l(\alpha)(y+_B z,x_2,\ldots, x_l)= \Sum_{i,j} G_{i+j+l-1}(\alpha\cdot  a_{i,j}^A)(y^{\times i}, z^{\times j},x_2,\ldots,x_l)$, for $l>0$,
where $a_{i,j}^A$ are the coefficients of the formal group law of $A^*$ and the sum $y+_B z$ is taken with respect to the
formal group law of $B^*$ (here $t^{\times i}$ denotes $i$ copies of $t$).
 \end{enumerate}
\end{theorem}

Here $B^*(k)[[x_1,\ldots, x_n]]_{(m)}$ is the subgroup in $B^*(k)[[x_1,\ldots, x_n]]$ consisting of all homogeneous degree $m$
power series (all the $x_i$'s have degree $1$).

The functions $G_l$ are determined by the operation $G$ as follows (see \cite[\S 5]{vishik19}). Write $L_i$ for the pull-back
of the canonical line bundle on $\P^\infty$ with respect to the $i$-th projection
$(\P^\infty)^{l}\to\P^\infty$. Then
\begin{equation}\label{c1eq}
G_l(\alpha)\(c_1^B(L_1),\ldots, c_1^B(L_l)\)=G\(\alpha\cdot c_1^A(L_1)\cdot\ldots\cdot c_1^A(L_l)\),
\end{equation}
where $c_1$ is the first Chern class.

\begin{remark}
Theorem \ref{vishik} was proved in \cite[Theorem 6.2]{vishik19} in the case $R=\Z$. The general case readily follows.
Indeed, multiplication by an element $r\in R$ yields operations $r:A_R^n\to A_R^n$ and $r:B_R^m\to B_R^m$. An additive operation
$G:A_R^n\to B_R^m$ is $R$-linear if and only if $G\circ r=r\circ G$ for all $r\in R$. The latter is equivalent to the equality
$G_l\circ r=r\circ G_l$ for all $l$, i.e., that all $G_l$ are $R$-linear.
\end{remark}

\begin{example}\label{ad}(see \cite[\S 6.3]{vishik19})
Let $A^*$ be a cohomology theory of rational type and $m\in\Z$. Consider the power series
$[m](x):=x+_A\ldots +_A x\in A^*(k)[[x]]$ ($m$ times). The \emph{Adams operation} $\Psi_m^A\in \Op^{*,*}_{R}$
is determined by $(G_l)_{l\geq 0}$, where $G_l$ is multiplication by the power series $[m](x_1)\cdot\ldots \cdot [m](x_l)$ ($l$ factors), in particular, $G_0$ is the identity.
The Adams operations satisfy the relations
\[
\Psi_k^A\circ \Psi_m^A =\Psi_{km}^A = \Psi_m^A\circ \Psi_k^A
\]
for all $k$ and $m$.
\end{example}

\subsection{Operations in connective $K$-theory}

We would like to determine the $R$-module $\Op_R^{n,m}$ of
all $R$-linear operations $G:\CK_R^n\to \CK_R^m$ for any pair of integers $n$ and $m$.
By Theorem \ref{vishik}, $G$ is given by a collection of power series $G_l(\alpha)\in R[t][[x_1,\ldots,x_n]]_{(m)}$, where
$\alpha\in \CK_R^{n-l}(k)$ and $l\geq 0$, satisfying conditions of the theorem. The group $\CK_R^{n-l}(k)$
is trivial if $l< n$ and $\CK_R^{n-l}(k)=R\cdot t^{l-n}$ otherwise. (Recall that $t$ has degree $-1$.)
In the first case $G_l(\alpha)=0$ and in the latter case the power series
$G_l(\alpha)$ are uniquely determined by $G_l(t^{l-n})$. We will simply write $G_l$ for $G_l(t^{l-n})$.

If $l\geq \max(1,n)$, condition $(3)$ in Theorem \ref{vishik} reads as follows (here $\bar z$ denotes $z_2,...,z_l$):
\[
G_l(x+y-txy,\bar z)=G_l(x,\bar z)+G_l(y,\bar z)-G_{l+1}(x,y,\bar z).
\]
In other words,
\begin{equation}\label{derive}
G_{l+1}=-\partial_t G_l,
\end{equation}
where the derivative $\partial_t$ is taken with respect to $F_{\CK}(x,y)=x+y-txy$. Thus, $G_{l+1}$ is uniquely determined by $G_l$.

If $n>0$, the operation $G$ yields the double-symmetric power series $G_n\in R[t][[x_1,\ldots,x_n]]_{(m)}$ that is divisible by
$x_1\cdot\ldots\cdot x_n$. Conversely, if $H \in R[t][[x_1,\ldots,x_n]]_{(m)}$ is a double-symmetric power series divisible by
$x_1\cdot\ldots\cdot x_n$, then setting $G_{n+i}:=(-1)^{i}\partial_t^i(H)$ for all $i\geq 0$, we get a sequence of power series that determines an $R$-linear operation $G$
(see Observation \ref{observe}).

If $n\leq 0$ the operation $G$ is determined by $G_0\in R[t]_{m}$ and power series
$G_1\in R[t][[x]]_{m}$ that is uniquely determined by $(G_1)|_{t=1}\in x^{\max{(1,m)}}R[[x]]$.
If $m> 0$ then $G_0=0$, otherwise $G_0\in R\cdot t^{-m}$ and
we can combine $G_0$ and $G_1$ together into the power series $H=(G_0-G_1)|_{t=1}\in R[[x]]$.

If $L\in R[t][[x_1,\ldots,x_n]]_{(m)}$, then $v(L|_{t=1})\geq m$. Conversely,
for every $J\in R[[x_1,\ldots,x_n]]$ with $v(J)\geq m$,
there is a unique homogeneous power series $L\in R[t][[x_1,\ldots,x_n]]$ of degree $m$ such that $L|_{t=1}=J$.
If $L$ is double-symmetric and divisible by $x_1\cdot\ldots\cdot x_n$, then so is $L|_{t=1}$
(with respect to the derivative $\partial$ given by the formal group law $x+y-xy$) and conversely.

We have proved the following statement.

\begin{proposition}\label{description2}
Let $R$ be a commutative ring and let $n$ and $m$ be two integers. An $R$-linear operation $G:\CK_R^n\to \CK_R^m$ is determined
by
\begin{enumerate}
  \item A power series $H\in x^{\max{(0,m)}}R[[x]]$ if $\underline{n\leq 0}$. In this case $G_0=H(0)\cdot t^{-m}$ and $G_1\in xR[t][[x]]_{(m)}$
  is a unique homogeneous power series such that $H=(G_0-G_1)|_{t=1}$ and $G_l=(-1)^{l-1}\partial_t^{l-1}(G_1)$ for $l>1$,
  \item A double-symmetric power series $J\in R[[x_1,\ldots, x_n]]$ divisible by $x_1\cdot\ldots\cdot x_n$ such that $v(J)\geq m$ if
$\underline{n>0}$.
  In this case $G_l=0$ for $l=0,\ldots, n-1$ and $G_n\in R[t][[x_1,\ldots,x_n]]_{(m)}$ is a unique homogeneous power series such that
  $G_n|_{t=1}=J$ and $G_l=(-1)^{l-n}\partial_t^{l-n}(G_n)$ for $l>n$.
\end{enumerate}
\end{proposition}

Let $R$ be a commutative ring that is torsion free
as abelian group.
Define an $R$-module homomorphism (see Definition \ref{qn2})
\[
\lambda_{n,m}:{\cQ}_R^{n,m}\to\Op_R^{n,m}
\]
as follows. If $n\leq 0$, $\lambda_{n,m}(H)$ for $H\in {\cQ}^{n,m}=x^{\max{(0,m)}}\cdot R[[x]]$
is the operation given by Proposition \ref{description2}(1).
If $n>0$, $\lambda_{n,m}(H)$ for $H\in {\cQ}^{n,m}$ is the operation given by the polynomial
$J=(-1)^{n}\partial^{n-1}(H)$ as in Proposition \ref{description2}(2).

The following theorem determines the $R$-module of operations $\Op_R^{n,m}$ in terms of the modules ${\cQ}_R^{n,m}$
of power series in one variable.

\begin{theorem}\label{mainone}
Let $R$ be a commutative ring that is torsion free
as abelian group and $K=R\tens \Q$.
The homomorphisms $\lambda_{n,m}$ yield an $R$-linear isomorphisms between $\Op_R^{n,m}$ and
the factor module of ${\cQ}_R^{n,m}$ by the $K$-subspace spanned by
$\lg_{i}(x)$, $i=1,\ldots, n-1$. In particular, $\Op_R^{n,m}\simeq x^{\max{(0,m)}}\cdot R[[x]]$ if $n\leq 0$ and
$\Op_R^{1,m}\simeq x^{\max{(1,m)}}\cdot R[[x]]$.
\end{theorem}

\begin{proof}
The surjectivity of $\lambda_{n,m}$ follows from Propositions \ref{allequivalent} and \ref{description2}.
The kernel of $\lambda_{n,m}$ is determined in Proposition \ref{kernel}.
\end{proof}

\begin{corollary}\label{corcor}
The map $\lambda_{n,m}$ yields an isomorphism (see Definition \ref{qn})
\[
\cQ_R^n\cap x^{\max{(0,n,m)}}\cdot K[[x]]\iso \Op_R^{n,m}.
\]
\end{corollary}
\begin{proof}
The case $m\leq n$ follows from the theorem. Otherwise, by Observation \ref{valuation}, $v(\partial^{n-1} x^i)=i$ for all $i\geq n$.
\end{proof}

Let $n,m\in\Z$ and $i,j$ non-negative integers. We define an $R$-linear homomorphism
\[
{\cQ}_R^{n,m}\to {\cQ}_R^{n+i,m-j}
\]
as follows. If $n\leq 0, m\leq 0$ and $n+i>0$ the map
\[
{\cQ}_R^{n,m}=R[[x]]\to xR[[x]]\hookrightarrow {\cQ}_R^{n+i,m-j}
\]
takes $H$ to $\partial^0(H)=H-H(0)$. Otherwise, ${\cQ}_R^{n,m}\subset {\cQ}_R^{n+i,m-j}$,
and the map we define is the inclusion.

Multiplication by $t^k$ yields an operation $\CK_R^{*+k}\to \CK_R^*$ and therefore, the homomorphisms
$\Op_R^{n,m}\to \Op_R^{n+i,m-j}$ for all $i,j\geq 0$.

\begin{proposition}\label{compat}
The diagram
\[
\xymatrix{
{\cQ}_R^{n,m}  \ar[d]_{\lambda_{n,m}} \ar[r] & {\cQ}_R^{n+i,m-j}  \ar[d]^{\lambda_{n+i,m-j}}  \\
\Op_R^{n,m} \ar[r]    &  \Op_R^{n+i,m-j},
}
\]
is commutative.
\end{proposition}

\begin{proof}
The case $i=0$ follows directly from the definition. It remains to consider the case $i=1$ and $j=0$.

Suppose first that $n>0$. Let $H\in {\cQ}_R^{n,m}\subset {\cQ}_R^{n+i,m-j}$ and $G=\lambda_{n,m}(H)\in \Op_R^{n,m}$.
In particular, $G_n|_{t=1}=(-1)^{n-1}\partial^{n-1}(H)$.
Denote by $G'$ the image of $G$ in $\Op_R^{n+1,m}$. Write $L_i$ for the pull-back
of the canonical line bundle on $\P^\infty$ with respect to the $i$-th projection
$(\P^\infty)^{n+1}\to\P^\infty$. The power series $G'_{n+1}$ is determined by the equality (see (\ref{c1eq}))
\begin{align*}
G'_{n+1}(c_1(L_1),\ldots,c_1(L_{n+1}))&=G'(c_1(L_1)\cdot\ldots\cdot c_1(L_{n+1}))\\
&=G(tc_1(L_1)\cdot\ldots\cdot c_1(L_{n+1}))\\
&=G_{n+1}(t)(c_1(L_1),\ldots, c_1(L_{n+1}))\\
&=G_{n+1}(c_1(L_1),\ldots, c_1(L_{n+1})),\\
\end{align*}
hence $G'_{n+1}=G_{n+1}$. It follows from (\ref{derive}) that
\[
G'_{n+1}|_{t=1}=G_{n+1}|_{t=1}=-(\partial_t G_n)|_{t=1}=-\partial(G_n|_{t=1})=-\partial((-1)^{n}\partial^{n-1}(H))=(-1)^{n+1}\partial^{n}(H),
\]
and therefore, $G'=\lambda_{n+1,m}(H)$.

If $n<0$ or if $n=0$  and $m>0$ we have ${\cQ}_R^{n,m}\subset {\cQ}_R^{n+1,m}$
and the statement follows immediately from the definitions. It remains to consider the case $n=0$ and $m\leq 0$.
Let $H\in {\cQ}_R^{0,m}=R[[x]]$ and $G=\lambda_{0,m}(H)\in \Op_R^{0,m}$. In particular, $H=(G_0-G_1)|_{t=1}$.
Denote by $G'$ the image of $G$ in $\Op_R^{1,m}$. A computation as above shows that $G'_1=G_1$. Hence
\[
G'_{1}|_{t=1}=G_{1}|_{t=1}=-(H-H(0))
\]
Therefore, $G'=\lambda_{1,m}(H-H(0))$ and $H-H(0)$ is the image of $H$ in ${\cQ}_R^{1,m}$.
\end{proof}

Corollary \ref{corcor} and Proposition \ref{compat} yield:

\begin{corollary}\label{corcor2}
If $m\leq n$ then the map $\Op_R^{n,n}\to \Op_R^{n,m}$ is an isomorphism.
\end{corollary}

In particular, there is a canonical ring homomorphism
\[
\Op_R^{n,n}\to \Op_R^{n+1,n}\iso \Op_R^{n+1,n+1}.
\]

\begin{example}
Note that the identification $\Op_R^{0,0}=R[[x]]$ is not a ring isomorphism. The corresponding ring structure
on $R[[x]]$ will be described in Section \ref{compos}.
The natural surjective homomorphism
\[
R[[x]]=\Op_R^{0,0}\to \Op_R^{1,1}=xR[[x]]
\]
takes a power series $G(x)$ to $G(x)-G(0)$. Its kernel is generated by $1$.
The complementary operation $G(x)\mapsto G(0)$ on $\CK^0=K_0$ is an idempotent that takes the class of a vector bundle $E$
to $\rank(E)\cdot 1$, where $1$ is the identity in $K_0$. In particular, we get a natural $R$-algebra isomorphism
$\Op_R^{0,0}\simeq R\times \Op_R^{1,1}$.
\end{example}

\subsection{Adams operations}\label{adams2}
Let $R$ be a torsion free ring. We define the composition
\[
\Ad_n:R[[x]]\to \cQ_R^n\xra{\lambda_{n,n}} \Op_R^{n,n},
\]
where the first map is the identity if $n\leq 0$ and it is the composition of the projection
$\partial^0:R[[x]]\to xR[[x]]$ and the inclusion of $xR[[x]]$ into $\cQ_R^n$.
The image of $\Ad_n$ is denoted $\Op^{n,n}_{R,cl}$ and called the submodule of \emph{classical operations}.

If $n\leq 0$, we have $\Op^{n,n}_{R,cl}=\Op_R^{n,n}=R[[x]]$.
If $n \geq 1$ it follows from Lemma \ref{divdiv} and Theorem \ref{mainone} that in the case $R$ has no nontrivial $\Z$-divisible elements
(for example. $R=\Z$ or $\widehat\Z$),
the restriction of $\Ad_n$ on $xR[[x]]$ is injective and therefore, $\Op^{n,n}_{R,cl}\simeq xR[[x]]$.

Let $m$ be an integer. In the notation of the Example \ref{ad}, $[m](x)=(1-(1-tx)^m)/t$.
In view of Proposition \ref{description2}, the Adams operations $\Psi_m\in \Op^{n,n}_{R,cl}$
are defined by
\begin{equation}\label{adamsoper}
\Psi_m=\Ad_n((1-x)^m).
\end{equation}
Since the power series $(1-x)^m$ generate $R[[x]]$ as topological $R$-module in the $x$-adic topology,
the group of classical operations $\Op^{n,n}_{R,cl}$
is topologically generated by the Adams operations.

By Proposition \ref{compat}, the operations $\Psi_k$ are compatible with the canonical homomorphisms
$\Op_{R}^{n,n}\to  \Op_{R}^{n+1,n+1}$.

For every $k\geq 0$ consider additive operations
$\Upsilon_k=\Sum_{i=0}^{k}(-1)^{i}\mybinom[0.8]{k}{i}\Psi_i$. Then $\Upsilon_k=\lambda_{n,n}(x^k)$ if $k\geq 0$.
Recall that $\Upsilon_0=0$ if $n\geq 1$.
It follows that the $R$-module
$\Op^{n,n}_{R,cl}$ consists of all linear combinations $\sum_{k\geq 0}\alpha_k\cdot\Upsilon_k$ with $\alpha_k\in R$
(cf. \cite[Theorem 6.8]{vishik19}). If $R$ has no nontrivial $\Z$-divisible elements,  the coefficients $\alpha_k$,
(where $k\geq 0$ if $n\leq 0$ and $k\geq 1$ if $n\geq 1$) are uniquely
determined by the operation.

\subsection{Operations over $\widehat\Z$}

In Section \ref{firstfirst} we determined the modules ${\cQ}_R^{n}$ over the ring $R=\widehat\Z$.
Theorems \ref{description} and \ref{mainone} yield:

\begin{theorem}\label{complete}
There are canonical isomorphisms
\[
\Op^{n,n}_{\widehat\Z}=\Op_{\widehat\Z,cl}^{n,n}\simeq
\left\{
  \begin{array}{ll}
   \widehat\Z[[x]], & \hbox{if $n\leq 0$;} \\
    x\widehat\Z[[x]], & \hbox{if $n\geq 1$}.
  \end{array}
\right.
\]
\end{theorem}
In particular, the natural map $\Op_{\widehat\Z}^{n,n}\to \Op_{\widehat\Z}^{n+1,n+1}$ is an isomorphism
for all $n\geq 1$.

It follows from Theorem \ref{mainone} that for any two integers $n$ and $m$,
\[
\Op_{\widehat\Z}^{n,m}\simeq
\left\{
  \begin{array}{ll}
   x^{\max{(0,m)}}\cdot \widehat\Z[[x]], & \hbox{if $n\leq 0$;} \\
   \{G\in x\widehat\Z[[x]]\ |\ v(\partial^{n-1}(G))\geq m\}, & \hbox{if $n\geq 1$.}
  \end{array}
\right.
\]

\subsection{Operations over $\Z$}

Now we turn to the case $R=\Z$ and for simplicity write $\Op^{n,m}$ for $\Op_{\Z}^{n,m}$.

Corollary \ref{corcor} implies that
the natural homomorphism $\Op^{n,m}\to \Op^{n,m}_{\widehat\Z}$ is injective. In particular, we can identify
$\Op^{n,n}$ with a subgroup of $\Op_{\widehat\Z}^{n,n} =x \widehat\Z[[x]]$ for all $n\geq 1$, so we have a sequence of subgroups
\[
\Op^{1,1}\subset \Op^{2,2}\subset\ldots\subset\Op^{n,n}\subset\ldots\subset x \widehat\Z[[x]].
\]

Recall (Theorem \ref{mainone}) that
$\Op^{n,m}\simeq x^{\max{(0,m)}}\cdot \Z[[x]]$ if $n\leq 0$ and $\Op^{n,m}\simeq \Op^{n,n}$ if $m\leq n$
by Corollary \ref{corcor2}.

Let $m\geq n\geq 1$. By Theorem \ref{mainone},  we can identify $\Op^{n,m}$ with the factor group of $\cQ^{n,m}$
by the subgroup $\Sum_{r=1}^{n-1}\Q\cdot\lg_{r}(x)$. It follows that the map $\rho_n$ in (\ref{sesimp}) yields
a homomorphism
\[
\Op^{n,m}\to (\widehat\Q/\Q)^{n-1}=(\widehat\Z/\Z)^{n-1}.
\]
By the proof of Lemma \ref{approxim}, this map is surjective.
Its kernel is denoted $\Op^{n,m}_{cl}$ and called the subgroup of \emph{classical operations}.
In the case $n=m$ this group coincides with the group of classical operation defined earlier.
In view of
Corollary \ref{corcor}, $\Op^{n,m}_{cl}$
is identified with the group $\(\Coprod_{r=1}^{n-1}\Q\cdot\lg_{r}(x) +x\Z[[x]]\)\cap x^{m} \Q[[x]]$.

We view the group $x\Z[x]_{\leq m-1}$
of integral polynomials of degree at most $m-1$ as a lattice in the $\Q$-space $x\Q[[x]]/(x^m)$. Denote by $\cL^{n,m}$ the
intersection of $x\Z[x]_{\leq m-1}$ with the image in $x\Q[[x]]/(x^m)$ of the space $\Coprod_{r=1}^{n-1}\Q\cdot\lg_{r}(x)$.
Then $\cL^{n,m}$ is a
subgroup of $x\Z[x]_{\leq m-1}$ of rank $n-1$.

We get the following description of the group of classical operations:
\[
\Op^{n,m}_{cl}= \cL^{n,m}\oplus x^m \Z[[x]].
\]
If $m=n\geq 1$, the map of $\Q$-spaces is an isomorphism and $\cL^{n,n}=x\Z[x]_{\leq n-1}$. It follows that
\[
\Op^{n,n}_{cl}= x \Z[[x]].
\]
Recall that $\Op^{n,n}_{cl}= \Op^{n,n} = \Z[[x]]$ if $n\leq 0$ and $\Op^{n,m}=\Op^{n,n}$ if $m\leq n$.

We summarize our results in the following statement.

\begin{theorem}\label{integoper}
The natural homomorphism $\Op^{n,m}\to \Op^{n,m}_{\widehat\Z}$ is injective.
For any integers $m\geq n\geq 1$ there is an exact sequence
\[
0\to \Op^{n,m}_{cl}\to \Op^{n,m} \to (\widehat\Z/\Z)^{n-1} \to 0,
\]
where $\Op^{n,m}_{cl}= \cL^{n,m}\oplus x^m \Z[[x]]$. Moreover,
$\Op^{n,n}_{cl}=x \Z[[x]]$.
\end{theorem}

\begin{remark}
Similar arguments yield the following formula for $m\geq n\geq 1$:
\[
\Op^{n,m}_{\widehat{\Z}}=\cL^{n,m}_{\widehat{\Z}}\oplus x^m \widehat\Z[[x]],
\]
where $\cL^{n,m}_{\widehat{\Z}}=\cL^{n,m}\tens\widehat\Z$.
\end{remark}

\subsection{Composition}\label{compos}

The $R$-module homomorphism $\Ad_n:R[[x]]\to \Op_R^{n,n}$ is not a ring homomorphism. In this section we
introduce a new product on $R[[x]]$ so that $\Ad_n$ becomes an $R$-algebra homomorphism.

Let $H,H'\in R[[x]]$, write $H'=\Sum_{i\geq 0} a_ix^i$ and define the \emph{composition} in $H$ and $H'$ by the formula
\[
H\circ H'=a_0\cdot H(0)+\Sum_{i\geq 1} (-1)^{i}a_i\cdot(\partial^{i-1}H)(x^{\times i}).
\]
The composition $\circ$ is distributive in $H$ and $H'$ with respect to addition. (Note that
the usual substitution of power series is only one-sided distributive.)
The polynomial $1-x$ is the identity for the composition: $(1-x)\circ H=H=H\circ (1-x)$ for all $H$.
We view $R[[x]]$ as an $R$-algebra with product given by the composition.



\begin{lemma}\label{compatibility}
The maps $\Ad_n:R[[x]]\to  \Op_R^{n,n}$ are $R$-algebra homomorphisms.
\end{lemma}
\begin{proof}
In view of Proposition \ref{compat} it suffices to consider the case $n=0$.
Let $H,H'\in R[[x]]$ and write $H'=\Sum_{i\geq 0} a_ix^i$. If $G_0,G_1,\ldots\in R[t][[x]]$
is the sequence of power series corresponding to $\Ad_0(H)$ (see Proposition \ref{description2}), then
$G_0=H(0)\in R$, $H=(G_0-G_1)|_{t=1}$ and $G_i=(-1)^{i-1}\partial_t^{i-1}(G_1)$ for $i>1$.
Note that $G_1(t,x)=-H(tx)+H(0)$.

Write $L$ for the canonical line bundle on $\P^\infty$. By (\ref{c1eq}) and (\ref{derive}),
\[
\Ad_0(H)(c_1(L)^i)=G_i(c_1(L)^{\times i})=(-1)^{i-1}(\partial_t^{i-1}G_1)(c_1(L)^{\times i})=(-1)^{i}(\partial^{i-1}H)(tc_1(L)^{\times i}).
\]
Therefore, we have
\begin{align*}
(\Ad_0(H)\circ \Ad_0(H'))(c_1(L))&=-\Ad_0(H)(\Sum_{i\geq 1} a_i c_1(L)^i) \\
&=-\Sum_{i\geq 1} a_i(\Ad_0(H))(c_1(L)^i) \\
& = \Sum_{i\geq 1} (-1)^{i-1}a_i\cdot(\partial^{i-1}H)(tc_1(L)^{\times i}).
\end{align*}
On the other hand, write $H\circ H'=(G''_0-G''_1)|_{t=1}$, where $G''_0=a_0\cdot H(0)$ and
\[
G''_1=\Sum_{i\geq 1} (-1)^{i-1}a_i\cdot(\partial^{i-1}H)(tx^{\times i}).
\]
It follows that
\[
\Ad_0(H\circ H')(c_1(L))=G''_1(c_1(L))=\Sum_{i\geq 1} (-1)^{i-1}a_i\cdot(\partial^{i-1}H)(tc_1(L)^{\times i})=
(\Ad_0(H)\circ \Ad_0(H'))(c_1(L)).
\]
If $r\in R=\CK^0_R(k)$, then
\[
\Ad_0(H\circ H')(r)=G''_0\cdot r=a_0\cdot H(0)\cdot r=\Ad_0(H)(a_0\cdot r)=(\Ad_0(H)\circ \Ad_0(H'))(r).
\]
Overall, $\Ad_0(H\circ H')=\Ad_0(H)\circ \Ad_0(H')$.
\end{proof}


The polynomials $A_m:=(1-x)^m$ satisfy $\Ad_n(A_m)=\Psi_m$ in $\Op_R^{n,n}$.
It follows from Lemma \ref{compatibility} and Example \ref{ad} that
\[
A_k\circ A_m =A_{km} =A_m\circ A_k
\]
for all $k$ and $m$.

\begin{proposition}\label{idempotents}
Let $R$ be a commutative ring and
$K$ a $\Q$-algebra. Then
\begin{enumerate}
  \item The composition $\circ$ in $R[[x]]$ is commutative.
  \item The power series $\lg_r(x)\in K[[x]]$, $r\geq 0$, are orthogonal idempotents that partition the identity,
  that is, $\lg_n(x)\circ\lg_m(x)=\delta_{n,m}\cdot\lg_n(x)$ and $1-x=\Sum_{r\geq 0}\lg_r(x)$.
\end{enumerate}
\end{proposition}

\begin{proof}
$(1)$ It follows from the definition that the power series $x^n\circ G$ and $G\circ x^n$
are contained in $x^n R[[x]]$ for all $n$ and $G$. Let $H,G\in R[[x]]$. Fix an integer $n>0$ and write
$H=H_1+H_2$ and $G=G_1+G_2$, where $H_1$ and $G_1$ are linear combinations of the Adams polynomials
$A_i$ and $H_2,G_2\in x^nR[[x]]$. As $H_1$ and $G_1$ commute, the remark above yields
$H\circ G-G\circ H\in x^nR[[x]]$. Since this holds for all $n$, we have $H\circ G=G\circ H$.

\smallskip

$(2)$ The iterated derivative $\partial^i (\lg_n(x))$ is zero if $i\geq n$ and
\[
(\partial^{n-1} \lg_n)(x_1,\ldots, x_n)=\Prod_{i=1}^{n}\log(1-x_i).
\]
It follows that $\lg_n(x)\circ x^m=0$ if $m>n$ and
\[
\lg_n(x)\circ x^n=(-1)^{n}(\partial^{n-1} \lg_n)(x^{\times n})=(-1)^{n}(\log(1-x))^n=(-1)^{n}n!\lg_n(x).
\]
This calculation together with the first part of the proposition and the fact that the lowest term of $\lg_r(x)$ is $x^r/r!$ show that the power
series $\lg_r(x)$ are orthogonal idempotents.

Finally, $\sum_{n\geq 1}\lg_n(x)=e^{\lg_1(x)}=(1-x)$.
\end{proof}

Since $\lg_r(x)$ are orthogonal idempotents which form a topological basis of the power series ring, from continuity and distributivity of $\circ$ we obtain that our composition is associative.

Theorems \ref{complete} and \ref{integoper} together with Proposition \ref{idempotents} yield the following corollary.

\begin{corollary}
The rings $\Op^{n,n}_{\widehat\Z}$ and $\Op^{n,n}$ are commutative.
\end{corollary}

Let $K$ be a $\Q$-algebra.
We view $K[[x]]$ as a ring with respect to addition and composition.
Let $G\in K[[x]]$ and write $G=\Sum_{i\geq 0}a_i\lg_i$ for (unique) $a_i\in K$. Denote as $K^{[n,\infty)}$ the ring of $K$-sequences, parametrized by integers $\geq n$ under point-wise operations.
It follows from Proposition \ref{idempotents} that the map
\begin{equation}\label{pass}
b:K[[x]]\to K^{[0,\infty)},
\end{equation}
taking $G$ to the sequence $(a_i)_{i\geq 0}$ is a ring isomorphism.
It takes $x^n K[[x]]$ onto $K^{[n,\infty)}$ for every $n$.

\begin{example}\label{adams}
The image of the polynomial $A_m(x)=(1-x)^m$ is equal to $(1,m,m^2,\ldots)$.
Indeed substituting $y=\log(1-x)$ in the equality $e^{my}=\Sum_{i\geq 0} \frac{m^iy^i}{i!}$
yields $A_m(x)=\Sum_{i\geq 0} m^i \lg_i(x)$.
\end{example}

\subsection{Topology}
In this section we introduce three topologies on $\widehat{\zz}[[x]]$.
\begin{proposition}\label{operoper}
\label{oper-filtr}
 Let $G\in{\Op}_{R}^{n,n}$ and $m\geq n$.
 The following conditions are equivalent:
 \begin{itemize}
  \item[$(1)$] $G\in\Im({\Op}_{R}^{n,m}  \to {\Op}_{R}^{n,n})$;
  \item[$(2)$] $G$ is zero on every smooth variety of dimension $<m$.
 \end{itemize}
\end{proposition}

\begin{proof}
 $(1)\Rightarrow (2)$ Since $\CK^m_{R}(X)=0$, for any variety $X$ of dimension $<m$, the operation $G$ is zero on $X$.

 $(2)\Rightarrow (1)$ Let $n\geq 1$. By Proposition \ref{description2}(2), the operation $G$ is given by a double-symmetric
power series $H(x_1,...,x_n)\in R[[x_1,...,x_n]]_{(n)}$ such that $H=(G_n)|_{t=1}$.
We need to prove that $v(H)\geq m$. We will show that
any monomial  $\ov{x}^{\ov{r}}=x_1^{r_1}\cdot\ldots\cdot x_n^{r_n}$ of $H$ with  $\sum_i r_i<m$ is zero.

Consider  $X_{\ov{r}}:=\prod_i\pp^{r_i}$. This is a variety of dimension  $<m$.
Write $x_i$ for the first Chern class in $\CK^1_R(X_{\ov{r}})$
of the pull-back
of the canonical line bundle on $\P^{r_i}$ with respect to the $i$-th projection
$X_{\ov{r}}\to\P^{r_i}$.
By formula (\ref{c1eq}),
\[
0=G(x_1\cdot\ldots\cdot x_n)=G_n(x_1,\ldots,x_n)\in\CK^n_R(X_{\ov{r}}).
\]
By Projective Bundle Theorem,
\[
\CK^n_R(X_{\ov{r}})=R[[x_1,\ldots, x_n]]/(x_1^{r_1+1},\ldots, x_n^{r_n+1}).
\]
Therefore the monomial $\ov{x}^{\ov{r}}$ of $H$ is trivial.

The case $n\leq 0$ follows similarly (and easier) from Proposition \ref{description2}(1).
\end{proof}

\begin{corollary}
Let $d\geq 0$ be an integer and $G\in \Op^{n,n}$. Then there is a $\Z$-linear combination $G'\in \Op^{n,n}$
of the Adams operations $\Psi_k$ with $k=0,\dots, d$ such that $G$ and $G'$ agree on $\CK^n(X)$ for all smooth varieties
$X$ of dimension $\leq d$.
\end{corollary}

\begin{proof}
By Lemma \ref{approxim}, applied to $m=d+1$,
there is a polynomial $G'\in \Z[x]$ of degree at most $d$ such that $G-G'\in \Sum_{0<r<n}\Q\cdot\lg_{r}(x)+x^{d+1}\Q[[x]]$.
Let $X$ be a smooth variety of dimension $\leq d$.
As $v(\partial^{n-1} (G-G'))\geq d+1$, in view of Theorem \ref{mainone}, $G-G'\in\Im({\Op}^{n,m}  \to {\Op}^{n,n})$.
Therefore,
by Proposition \ref{operoper}, $G-G'$ is trivial on $X$. Finally,
$G'$ is a linear combination of the Adams polynomials $A_k$ with $k=0,\dots, d$.
\end{proof}


\begin{definition}
We introduce three topologies on $\widehat{\zz}[[x]]$:
 \label{def-top3}
 \begin{itemize}
  \item[$\bullet$] $\tau_s$ is generated by the neighborhoods of zero $U_m$ consisting of power series divisible by $x^m$, for some
  $m\geq 0$, i.e., $\tau_s$ is the $x$-adic topology.
  \item[$\bullet$] $\tau_w$ is generated by the neighborhoods of zero $U_m+ V_N$, where $V_N$ consists of all power series divisible by some $N\in\nn$.
  \item[$\bullet$] $\tau_o$ is generated by the neighborhoods of zero $W_m$ consisting of power series, where the respective operation acts trivially on varieties of dimension $<m$.
 \end{itemize}
\end{definition}

Recall, that a topology $\ffi$ is {\it coarser} than the topology $\psi$, denoted $\ffi\leq\psi$,
if any set open with respect to $\ffi$ is also open with respect to $\psi$.

\begin{proposition}
 $\tau_w\leq\tau_o\leq\tau_s$.
\end{proposition}

\begin{proof}
Since $v(G(x))\geq m$ implies $v(\partial^{n-1}G(x))\geq m$ and hence
$G\in\Im({\Op}_{\widehat{\zz}}^{n,m}  \to {\Op}_{\widehat{\zz}}^{n,n})$
by Theorem \ref{mainone}. Therefore, it follows from Proposition \ref{oper-filtr} that
 $\tau_o\leq\tau_s$.

 The topology $\tau_w$ is generated by the neighborhoods of zero $U_{N,m}=(N,x^m)\subset\widehat\Z[[x]]$,
and $\tau_o$ is generated by the neighborhoods of zero $W_k=\{G\in \widehat\Z[[x]]\ |\ v(\partial^{n-1}(G))\geq k\}$
by Proposition \ref{operoper}.
We need to show that for every $N$ and $m$ there is $k$ with $W_k\subset U_{N,m}$.

We have similar compact (Hausdorff) topology $\tau_w$ on $\widehat\Z[[x_1,\ldots, x_n]]$ so that the map $\partial^{n-1}$ is
continuous in $\tau_w$.
Note that the map $\partial^{n-1}:\widehat\Z[[x]]\to \widehat\Z[[x_1,\ldots, x_n]]$ is injective and the induced map from
$\widehat\Z[[x]]$ to the image of $\partial^{n-1}$
is a homeomorphism (since the image of every closed subset is closed as $\widehat\Z[[x]]$ is compact
and the target is Hausdorff). In particular,
if $G_k\in \widehat\Z[[x]]$ is a sequence such that the sequence $\partial^{n-1}(G_k)$ converges to $0$, then the sequence
$G_k$ converges to $0$ in $\widehat\Z[[x]]$.

Now we prove that for every $N$ and $m$ there is $k$ with $W_k\subset U_{N,m}$. Assume on the contrary that
for every $k$ we can find $G_k\in W_k$,
but $G_k\notin U_{N,m}$. Then $\partial^{n-1}(G_k)$ converges to $0$, but $G_k$ does not converge to $0$ in $\widehat\Z[[x]]$,
a contradiction.
\end{proof}

\begin{observation}
 \begin{itemize}
  \item[$1)$] For $n=1$, $\tau_o=\tau_s$;
  \item[$2)$] For $n>1$, $\tau_w\neq\tau_o\neq\tau_s$.
 \end{itemize}
\end{observation}

\begin{proof}
 1) This follows from Proposition \ref{oper-filtr}, since $n=1$.

 2) For $n>1$, $W_m$ contains, in particular, all power series
 $\sum_i a_ix^i\in\widehat{\zz}[[x]]$, where $a_1=ia_i$, for all $0<i<m$, which is not contained in any $U_l$, for $l>1$. Thus,
 $\tau_o\neq\tau_s$.

 For $m>n\geq 1$, $W_m/U_m$ is a free $\widehat{\zz}$-module of rank
 $(n-1)$, while $(U_m+ V_N)/U_m$ is a free $\widehat{\zz}$-module
 of rank $(m-1)$. Hence, $\tau_w\neq\tau_o$.
\end{proof}

We view $\Op^{n,n}$ and ${\Op}^{n,n}_{\widehat{\zz}}$ as the topological rings for the topologies $\tau_w$, $\tau_o$ and $\tau_s$ respectively
via the inclusions $\Op^{n,n}\hookrightarrow \Op_{\widehat\Z}^{n,n}\hookrightarrow \widehat\Z[[x]]$.

Note that the $x$-adic topology $\tau_s$ can be defined on $R[[x]]$ for every $R$.

Consider the restriction $b:R[[x]] \to K^{[0,\infty)}$ of the map (\ref{pass}).
We view $K^{[0,\infty)}$ as a topological ring with the basis of neighborhoods of zero
given by the ideals $K^{[n,\infty)}$ for all $n>0$, so that the map is continuous.

\begin{proposition}\label{intimage}
The image of the map $b:R[[x]] \to K^{[0,\infty)}$ is contained in $R^{[0,\infty)}$.
\end{proposition}

\begin{proof}
By Example \ref{adams}, the image of the Adams polynomial $A_m$ under the map (\ref{pass}) is contained in
$R^{[0,\infty)}$. But the set of all linear combinations of Adams polynomials is dense in $R[[x]]$
in the topology $\tau_s$.
The statement follows since $R^{[0,\infty)}$ is closed in $K^{[0,\infty)}$.
\end{proof}

Proposition \ref{intimage} identifies the ring $\Op_{\widehat\Z}^{n,n}\subset\widehat\Z[[x]]$ with a subring of $\widehat\Z^{[n,\infty)}$
and $\Op^{n,n}$ with a subring of $\Z^{[n,\infty)}$ if $n\geq 0$. Indeed, if $n\geq 1$ the kernel of the composition
\[
{\cQ}_{\widehat\Z}^{n}\xra{\lambda_{n,n}}\hspace{-3mm}\to\Op_{\widehat\Z}^{n,n}\xra{b} {\widehat\Z}^{[0,\infty)} \to {\widehat\Z}^{[n,\infty)}
\]
is generated by $\lg_r$ with $0<r<n$ and all these logarithms are contained in the kernel of $\lambda_{n,n}$.

\smallskip

The ring $\Op^{n,n}$ is not a domain: we have $(\Psi_{1}+\Psi_{-1})(\Psi_{1}-\Psi_{-1})=0$.
Let
\[
e_{\pm}=\operatornamewithlimits{\textstyle\frac{1}{2}}(\Psi_{1}\pm \Psi_{-1})\in \Op^{n,n}\operatornamewithlimits{\textstyle[\frac{1}{2}]},
\]
so $e_{+}$ and $e_{-}$ are orthogonal idempotents and $e_{+} + e_{-}=1$.
There is an embedding
\[
\Op^{n,n}\hookrightarrow \Op^{n,n}\operatornamewithlimits{\textstyle[\frac{1}{2}]} = \Op^{n,n}\operatornamewithlimits{\textstyle[\frac{1}{2}]}e_{+}\times \Op^{n,n}\operatornamewithlimits{\textstyle[\frac{1}{2}]}e_{-}.
\]

\begin{proposition}
If $n\geq 1$ the rings $\Op^{n,n}\big[\frac{1}{2}\big]e_{\pm}$ are domains.
\end{proposition}

\begin{proof}
Recall that there is an injective ring homomorphism
\[
b: \Op^{n,n}\hookrightarrow \Z^{[1,\infty)}
\]
such that $b(\Psi_m)=(m,m^2,m^3,\ldots)$ for all $m$. In particular,
\[
b(e_{+})=(0,1,0,1,\ldots)\quad\text{and}\quad  b(e_{-})=(1,0,1,0,\ldots).
\]

\begin{lemma}\label{div}
Let $(a_1,a_2,\ldots)\in \Im(b)$. Then for any prime integer $p$, we have $a_i\equiv a_j$ modulo $p$
if $i\equiv j$ modulo $p-1$.
\end{lemma}

\begin{proof}
It suffices to prove the statement for $b(\Psi_m)$. We have $a_i-a_j=m^i-m^j=m^j(m^{i-j}-1)$. If $m$
is not divisible by $p$, then $m^{i-j}-1$ is divisible by $p$.
\end{proof}

Let $G\cdot H=0$ in $\Op^{n,n}$. Set $(a_1,a_2,\ldots)=b(G)$ and $(b_1,b_2,\ldots)=b(H)$. We have
$a_ib_i=0$ for all $i$. To prove the statement it suffices to show that if $a_i\neq 0$ for some $i$,
then $b_j=0$ for all $j\equiv i$ modulo $2$.

Choose an odd prime $p$ that does not divide $a_i$. By Lemma \ref{div}, $a_j$ is not divisible by $p$
for all $j$ such that $i\equiv j$ modulo $p-1$. In particular $a_j\neq 0$, hence $b_j=0$. Thus, we have proved that
$b_j=0$ for all $j\equiv i$ modulo $p-1$.

\begin{lemma}\label{div2}
There are infinitely many primes $q$ such that $\gcd(q-1,p-1)=2$.
\end{lemma}

\begin{proof}
Let $c$ be the odd part of $p-1$ (that is $(p-1)/c$ is a $2$-power). By Dirichlet, there are infinitely
many primes $q$ such that $q\equiv 3$ modulo $4$ and $q\equiv 2$ modulo $c$. Clearly, $\gcd(q-1,p-1)=2$
for such $q$.
\end{proof}

Let $j$ be such that $j\equiv i$ modulo $2$. We need to prove that $b_j=0$. Take any prime $q$ as in Lemma
\ref{div2}. There are positive integers $k$ and $m$ such that $t:=i+(p-1)k=j+(q-1)m$. We have proved that
$b_t=0$ since $t\equiv i$ modulo $p-1$. By Lemma \ref{div}, $0=b_t\equiv b_j$ modulo $q$, i.e., $b_j$
is divisible by $q$. We have proved that $b_j$ is divisible by infinitely many primes $q$,
hence $b_j=0$.
\end{proof}

\subsection{Operations in graded $K$-theory}\label{graded}
In this section we determine the $R$-module of all $R$-linear operations $G:K_{gr R}^n\to K_{gr R}^m$ for any pair of integers $n$ and $m$
denoted by $\Op^{n,m}_R(K^*_{gr})$.
Recall that $K_{gr R}^n=K_{gr R}^0\cdot t^{-n}=\CK^0_R\cdot t^{-n}$, hence by Theorem \ref{mainone},
we get:

\begin{corollary}
\[
\Op^{n,m}_R(K^*_{gr})=\Op^{0,0}_R(K^*_{gr})\cdot t^{m-n}=\Op^{0,0}_R(\CK^*)\cdot t^{m-n}=R[[x]]\cdot t^{m-n}.
\]
\end{corollary}

Recall that product operation in the ring $\Op^{n,m}_R(K^*_{gr})=R[[x]]$
is the composition $\circ$ (see Section \ref{compos}). Moreover, $R[[x]]$ is a (topological) bi-algebra over $R$ with
co-product defined by the rule $(1-x)^n\to (1-x)^n\tens (1-x)^n$ for all $n\geq 0$ that reads $\Psi\mapsto \Psi\tens\Psi$
in the language of operations.

Let us describe the dual bi-algebra $A$ (over $\Z$) of \emph{co-operations} as follows. Let $A$ be the subring
of the polynomial ring $\Q[s]$ consisting of all polynomials $f$ such that $f(a)\in\Z$ for all $a\in\Z$.
In particular, $\Z[s]\subset A$. The polynomials
\[
e_n:=\frac{1}{n!}\ (-s)(1-s)\ldots (n-1-s)=(-1)^n\mybinom[0.8]{s}{n}\in A
\]
for all $n\geq 0$ form a basis of $A$ as an abelian group. Consider a pairing
\[
A\otimes R[[x]]\to R,\quad\quad a\tens G\mapsto \langle a, G \rangle\in R,
\]
such that $\langle e_n, x^m \rangle=\delta_{n,m}$. This pairing identifies
$R[[x]]$ with the dual co-algebra for $A$ via the isomorphism
\[
\Hom_{\Z}(A,R)\iso R[[x]],
\]
taking a homomorphism $\alpha:A\to R$ to the power series $\sum_{n\geq 0} \alpha(e_n)x^n$.

\begin{lemma}\label{vish}
For every polynomial $f\in A$, we have $\langle f, (1-x)^m \rangle=f(m)$.
\end{lemma}

\begin{proof}
We may assume that $f=e_n$ for some $n$. Then
\[
\langle f, (1-x)^m \rangle=\langle e_n, (1-x)^m \rangle=(-1)^n\mybinom[0.8]{m}{n}=e_n(m)=f(m).\qedhere
\]
\end{proof}
The lemma shows that a co-operation $f$ evaluated at the Adams operation $\Psi_m$ is equal to $f(m)$.

It follows from Lemma \ref{vish} that
\[
\langle s^n, (1-x)^{km} \rangle=(km)^n=k^n\cdot m^n=\langle s^n, (1-x)^{k} \rangle\cdot  \langle s^n, (1-x)^{m} \rangle.
\]
As the composition in $R[[x]]$ satisfies $(1-x)^{k}\circ (1-x)^{m}=(1-x)^{km}$, the composition in $R[[x]]$
is dual to the co-product of $A$ taking $s^n$ to $s^n\tens s^n$ in $A\tens A$.

The equality
\[
\langle s^{i+j}, (1-x)^{m} \rangle=m^{i+j}=m^i\cdot m^j=\langle s^{i}, (1-x)^{m} \rangle\cdot\langle s^{j}, (1-x)^{m} \rangle
\]
shows that the product in $A$ is dual to the co-product in $R[[x]]$. Thus, the bi-algebra $R[[x]]$ of operations is dual to the
bi-algebra $A$ of co-operations.

\begin{remark}
The polynomial ring $\Z[s]$ is a bi-algebra with respect to the co-product $s\to s\tens s$. The dual bi-algebra over $R$
is $R^{[0,\infty)}$. The dual of the embedding $\Z[s]\to A$ is the homomorphism $b:R[[x]]\to R^{[0,\infty)}$
defined in Proposition \ref{intimage} since by Lemma \ref{vish},
\[
\langle s^{n},(1-x)^{m} \rangle=m^n=\langle s^{n},b((1-x)^{m}) \rangle
\]
as $b((1-x)^{m})=(1,m,\ldots,m^n,\ldots)$.
\end{remark}

\section{Multiplicative operations}\label{multoperat}

\begin{definition}
 \label{def-mult-oper}
 A \emph{multiplicative operation} $G:A^*\row B^*$ is a morphism of functors from $\smk$ to the category of rings.
That is, the ring structure is respected.
(We don't assume that $G$ is a graded ring homomorphism.)
\end{definition}

As was noticed in topology and then in the algebro-geometric context in \cite[2.7.5]{Panin04} there is a
functor from the category of oriented cohomology theories and their multiplicative operations to the category of formal group laws. Let us briefly describe this functor.

If $A^*$ and $B^*$ are oriented cohomology theories over $k$, to any multiplicative operation $G:A^*\row B^*$ one can assign the
morphism $(\ffi_G,\gamma_G): (A^*(k),F_A)\row (B^*(k),F_B)$ of the respective formal group laws, where $\ffi_G:A^*(k)\row B^*(k)$ is the restriction of $G$ to $\spec(k)$ and $\gamma_G(x)\in xB^*(k)[[x]]$ is defined by the condition:
\[
G(c_1^A(O(1)))=\gamma_G(c_1^B(O(1)))\in B^*(\pp^{\infty})=B^*(k)[[x]].
\]

In the algebro-geometric context, the power series $\gamma_G(x)/x$ was introduced in this generality in
\cite[Definition 2.5.1]{Panin04} and \cite{Smirnov07}
in order to state and prove a Riemann-Roch type theorems
\cite[Theorem 2.5.3, 2.5.4]{Panin04} for a multiplicative operation $G$. This series is called the {\it inverse Todd genus} of $G$.

The following theorem permits to reduce the classification of multiplicative operations to algebra.

\begin{theorem} {\rm(\cite[Theorem 6.9]{vishik19})}
 \label{mult-op-mor-FGL}
 Let $A^*$ be a theory of rational type and $B^*$ any oriented cohomology theory. Then the assignment $G\mapsto (\ffi_G,\gamma_G)$ is a bijection between the set of multiplicative operations $G:A^*\row B^*$ and the set of morphisms of formal group laws.
\end{theorem}

\begin{example}\label{twisted}
Let $R$ be either $\Z$, $\Z_p$ or $\widehat\Z$ and $b\in R$. The Adams operation $\Psi_b:\CK_R^*\to \CK_R^*$ is homogeneous
and multiplicative. The corresponding map $\ffi$ is the identity and $\gamma=\frac{1-(1-tx)^b}{t}$. If $c\in R^\times$, write
$\Psi^c_b$ for the homogeneous multiplicative \emph{twisted} Adams operation with $\ffi(t)=ct$ and $\gamma=\frac{1-(1-tx)^{bc}}{ct}$
(in particular, $\Psi^1_b=\Psi_b$). It follows from the equality $\Psi^c_b(tx)=\Psi^c_b(t)\Psi^c_b(x)=ct\cdot\gamma(x)=1-(1-tx)^{bc}$
that on $\op{CK}_R^n$ the operation $\Psi^c_b$ is equal to $c^{-n}\cdot\Psi_{bc}$. For any $c\in R$, let $\Psi^c_0$ be the
homogeneous multiplicative operation with $\ffi(t)=ct$ and $\gamma=0$. This operation is zero in positive degrees and is
equal to $c^n\cdot\rank$ on $\CK_R^{-n}=(K_0)_R$ for $n\geq 0$.

Write $\Theta$ for the multiplicative operation $\CK_R^*\to \CK_R^*$ which is identity on $\CK_R^0$, multiplication by $t^n:\CK_R^n\to \CK_R^0$
if $n\geq 0$ and the canonical isomorphism $\CK_R^n\to \CK_R^0$ (inverse to multiplication by $t^{-n}$) if $n\leq 0$. This operation
is not homogeneous and its image is $\CK_R^0$. Set $\widetilde\Psi^c_b:=\Theta\circ\Psi^c_b$. This is a multiplicative operation
with image in $\CK_R^0$. The corresponding function $\ffi(t)=c$ and $\gamma=\frac{1-(1-tx)^{bc}}{c}$.

Introduced operations satisfy the following relations
(use Theorem \ref{mult-op-mor-FGL}): $\Psi^{0}_{0}=\widetilde\Psi^{0}_{0}$ and
\begin{equation*}
 \Psi^{c}_{b}\circ\Psi^{e }_{d }=\Psi^{c e }_{b d };
 \hspace{5mm}
 \Psi^{c }_{b }\circ\wt{\Psi}^{e }_{d }=
 \wt{\Psi}^{e }_{c b d };
 \hspace{5mm}
 \wt{\Psi}^{c }_{b }\circ\Psi^{e }_{d }=
 \wt{\Psi}^{c e }_{b d };
 \hspace{5mm}
 \wt{\Psi}^{c }_{b }\circ\wt{\Psi}^{e }_{d }=
 \wt{\Psi}^{e }_{c b d }.
\end{equation*}
\end{example}

Over $\qq$ every formal group law is isomorphic to the additive
one. Hence, for every theory $C^*$, we have isomorphisms of formal group laws.
\begin{equation*}
\begin{split}
 &\xymatrix{
 (\id,\exp_C):(C^*\otimes_{\zz}\qq, F_C) \ar@/^0.7pc/[rr]& &
 (C^*\otimes_{\zz}\qq, F_{add}):(\id,\log_C) \ar@/^0.7pc/[ll]
 }
 \end{split}
\end{equation*}

Suppose that (in the context of Definition \ref{def-mult-oper}) the coefficient ring $B^*(k)$ of the target theory has no torsion.
Then the composition
$(\id, \exp_B)\circ (\ffi_G,\gamma_G)\circ (\id,\log_A)$
identifies the set of multiplicative operations $A^*\row B^*$
with a subset of morphisms of formal group laws
$(A^*\otimes_{\zz}\qq,F_{add})\row (B^*\otimes_{\zz}\qq, F_{add})$.
The latter morphism is defined by $(\psi,\gamma)$, where, in our case,
$\psi=\ffi\otimes_{\zz}\qq$, for some ring homomorphism
$\ffi=\ffi_G:A^*(k)\row B^*(k)$ and
$\gamma(x)=b\cdot x$, for some $b\in B^*(k)$.
In other words, $(\ffi_G,\gamma_G)=(\id,\log_B)\circ(\ffi_G,\gamma)\circ (\id,\exp_A)$. Then
$$
\gamma_G(x)=\ffi_G(\exp_A)(b\cdot \log_B(x)).
$$

\subsection{Multiplicative operations in $\CK$}

For $A^*=B^*=\op{CK}^*_{\widehat{\zz}}$ we have:
$A=B=\widehat{\zz}[t]$, $F_A=F_B=x+y-txy$ and
\[
\log_{\CK}(x)=\frac{\log(1-tx)}{t},\quad \exp_{\CK}(z)=\frac{1-e^{zt}}{t}.
\]

Note that a ring homomorphism $\varphi$ from $\widehat{\zz}[t]$ to a ring $T$ such that $\Bigcap_{n>0}nT=0$
is uniquely determined by $\varphi(t)$ in $T$ (such a choice is realised by a homomorphism, if $\widehat{\zz}$ can be mapped to $T$). Indeed, suppose that $\varphi$ and $\psi$ satisfy $\varphi(t)=\psi(t)$. For
any $f\in \widehat{\zz}[t]$ and $n>0$ write $f=g+nh$ for some $g\in {\zz}[t]$ and $h\in \widehat{\zz}[t]$.
Then $\varphi(g)=\psi(g)$ and hence $\varphi(f)-\psi(f)\in nT$. Since this holds for all $n>0$, we have $\varphi(f)-\psi(f)=0$
for all $f$.

Thus, the map $\ffi_G:\widehat{\zz}[t]\to \widehat{\zz}[t]$ is determined by
$\ffi_G(t)=c(t)\in\widehat{\zz}[t]$.
Let $b=b(t)\in\widehat{\zz}[t]$.
Note that any choice of $b(t)$ and $c(t)$ gives a morphism of rational formal group laws and so, a multiplicative operation
$G:\op{CK}^*_{\widehat{\zz}}\otimes_{\zz}\qq\row
\op{CK}^*_{\widehat{\zz}}\otimes_{\zz}\qq$
with
$$
\gamma_G(t,x)=\frac{1-(1-tx)^{\frac{b(t)c(t)}{t}}}{c(t)}=
\Sum_{n\geq 1}(-1)^{n-1}(tx)^n\frac{\binom{\frac{b(t)c(t)}{t}}{n}}{c(t)},
$$
which lifts to an operation
$\CK^*_{\widehat{\zz}}\row \CK^*_{\widehat{\zz}}$ if and only if the coefficients of our power series belong to $\widehat{\zz}$.
The coefficient at $x^n$ is
\begin{equation}\label{coeffi}
a_n=\displaystyle(-1)^{n-1}\frac{b(t)\prod_{k=1}^{n-1}(b(t)c(t)-kt)}{n!}.
\end{equation}
Denote as $b_p(t),c_p(t)$ the $\zz_p$-components of our
polynomials. If $\deg(b_p(t)c_p(t))>1$ for some $p$, the leading term of our $t$-polynomial will be clearly non-integral (for some $n$). Similarly, if
for some $p$, the constant term of $b_p(t)c_p(t)$ is non-zero,
then the smallest term of the $p$-component of our $t$-polynomial will be non-integral, for some $n$. Hence, the polynomial $b(t)c(t)$ is linear. Then, for a given prime $p$, either
$b_p(t)=b_p$ and $c_p(t)=c_pt$, or $b_p(t)=b_pt$ and
$c_p(t)=c_p$, for some $b_p,c_p\in\zz_p$. Then the $\zz_p$-
component of our coefficient is:
$$
(a_n)_p=(-1)^{n-1}t^m b_p\frac{\binom{b_pc_p-1}{n-1}}{n},\hspace{5mm}\text{where}\,\,
m=n-1,\,\,\text{or}\,\,m=n.
$$
If $b_p\neq 0$, then this will be integral for all $n$ if and only if $c_p\in\zz_p^{\times}$, while if $b_p=0$, then $c_p$ can
be an arbitrary element from $\zz_p$.
Let us denote the ($\zz_p$-components of) operations with $m=n-1$
as $\Psi^{c_p}_{b_p}$,
while the ones with $m=n$ as $\wt{\Psi}^{c_p}_{b_p}$ (see Example \ref{twisted}; we suppress
$p$ from notations). Here $\Psi^{c_p}_{b_p}$ respects the grading
on $\op{CK}^*_{\zz_p}$, while $\wt{\Psi}^{c_p}_{b_p}$ maps
$\op{CK}^*_{\zz_p}$ to $\op{CK}^0_{\zz_p}$. The pairs $(b_p,c_p)$
run over the set $(\zz_p\backslash 0)\times\zz_p^{\times}\cup\{0\}\times\zz_p$ and, in addition,
$\Psi^{0}_{0}=\wt{\Psi}^{0}_{0}$.

Thus, any multiplicative operation $G$ on
$\op{CK}^*_{\widehat{\zz}}$ splits into the
product $\times_p G_{(p)}$ of operations
on $\op{CK}^*_{\zz_p}$, where each $G_{(p)}$ is one of the
$\Psi^{c_p}_{b_p}$ or $\wt{\Psi}^{c_p}_{b_p}$. Let ${\cP}$
will be the set of prime numbers and $J\subset{\cP}$ be the
subset of those primes, for which $(b_p,c_p)\neq (0,0)$ and
$G_{(p)}$ is $\wt{\Psi}$.
Then the data $(J,b,c)$, where the $p$-components of
$b,c\in\widehat{\zz}$
are $b_p$ and $c_p$, determines our operation $G$.
Let us call it $\ppsi{J}{c}{b}$. Here $(J,b,c)$ runs over
all possible triples satisfying:
1) $b_p\neq 0$ $\Rightarrow$ $c_p\in\zz_p^{\times}$ and
2) $(b_p,c_p)=(0,0)$ $\Rightarrow$ $p\notin J$.

The operations $\ppsi{\emptyset}{1}{b}$ are
(non-twisted) Adams operations with $\ffi_G=\id$, which
naturally form a ring isomorphic to $\wzz$.
These operations commute with every other operation.
The operations $\ppsi{\emptyset}{c}{1}$ are invertible and
form a group 
isomorphic to $\wzz^{\times}$.
Below we will suppress $J=\emptyset$ from notations and will
denote the respective operations simply as $\Psi^c_b$.

The formulas in Example \ref{twisted} show that the monoid of multiplicative operations
is non-commutative.

\subsection{Multiplicative operations in $K_{gr}$ over $\Z$}

For $A^*=B^*=K^*_{gr}$ we have:
$A=B=\zz[t,t^{-1}]$, $F_A=F_B=x+y-txy$. Similar calculations as in the previous section
show that the coefficient $a_n$ in (\ref{coeffi}) will belong to $\zz[t,t^{-1}]$ for every $n$,
if and only if $b(t)c(t)$ is linear in $t$. Thus,
$c(t)=ct^l$, for $c=\pm 1$ and $l\in\zz$, and $b(t)=bt^{1-l}$,
for some $b\in\zz$.

Then the coefficient $a_n$ is
$$
(-1)^{n-1}t^{n-l}\frac{\binom{bc}{n}}{c}.
$$
Denote this operation as $\ppsi{l}{c}{b}$.
It scales the grading on $K^*_{gr}$ by the coefficient $l$.
So, only the operations $\ppsi{1}{c}{b}$ are homogeneous.

The case $c(t)=t$ and $b(t)=b$, that is, $\ppsi{1}{1}{b}$ corresponds to the Adams operation
$\Psi_b$ - see \cite[Sect. 6.3]{vishik19}. In this case $\ffi_G=id$.
The operation $\ppsi{-1}{1}{1}$ is an automorphism of order $2$
acting identically on $K^0_{gr}$ and mapping $t$ to $t^{-1}$.

We will omit $l$ and $c$ from the notation $\ppsi{l}{c}{b}$
when these will be equal to $1$.

\section{Stable operations}\label{stableoper}

The purpose of this section is to describe stable operations in $\CK$ and $K_{gr}$ with integral and $\widehat{\zz}$-coefficients. The spaces of such operations appear to have countable topological base
which we describe in Theorems \ref{poly3} and \ref{stable-op-Kgr}. We also describe stable multiplicative operations and show that these generate
additive ones only in the case of $\widehat{\zz}$-coefficients.

To be able to discuss stability of operations, we need the notion of a suspension. Following Voevodsky and Panin-Smirnov
\cite{Panin03,Smirnov07}
we can introduce the category of pairs $\smop$ whose objects are pairs
$(X,U)$, where $X\in\smk$ and $U$ is an open subvariety of $X$ -
see \cite[Def. 3.1]{vishik19}, with the smash product:
$$
(X,U)\wedge (Y,V):=(X\times Y,X\times V\cup U\times Y),
$$
and the natural functor $\smk\row\smop$ given by $X\mapsto (X,\emptyset)$.
Then suspension can be defined as:
$$
\Sigma_T(X,U):=(X,U)\wedge(\pp^1,\pp^1\backslash 0).
$$
Any theory $A^*$ extends from $\smk$ to $\smop$ by the rule:
$$
A^*((X,U)):=\op{Ker}(A^*(X)\row A^*(U)).
$$
Any additive operation $A^*\row B^*$ on $\smk$ extends
uniquely to an operation on $\smop$.

An element $\eps^A=c_1^A(O(1))\in A^*((\pp^1,\pp^1\backslash 0))$
defines an identification:
$$
\sigma_T^A:A^*((X,U))\stackrel{\cong}{\lrow} A^{*+1}(\Sigma_T(X,U)),
$$
given by $x\mapsto x\wedge\eps^A$.

\begin{definition}
 \label{susp-oper}
 For any additive operation $G:A^*\row B^*$ we define its \emph{desuspension}
 as the unique operation $\Sigma^{-1} G:A^*\row B^*$ such that
 $$
 G\circ\sigma_T^A=\sigma_T^B\circ\Sigma^{-1}G.
 $$
\end{definition}

\begin{definition}
\label{stable-operation}
A \emph{stable} additive operation $G:A^*\row B^*$ is the collection
$\{G^{(n)}|n\geq 0\}$ of operations $A^*\row B^*$ such that
$G^{(n)}=\Sigma^{-1}G^{(n+1)}$.
\end{definition}

\begin{proposition}
\label{mult-desusp}
 Suppose, $G:A^*\row B^*$ is a multiplicative operation with
 $\gamma_G(x)\equiv bx$ modulo $x^2$ for some $b\in B^*(k)$. Then
 $\Sigma^{-1}G=b\cdot G$.
\end{proposition}

\begin{proof}
 We have: $G(\sigma_T^A(u))=G(u\wedge\eps^A)=G(u)\wedge G(\eps^A)=
 G(u)\wedge (b\cdot\eps^B)=\sigma_T^B(b\cdot G(u))$.
\end{proof}

We call a multiplicative operation $G$ {\it stable} if the constant sequence
$(G,G,G,\ldots)$ is stable. By Proposition \ref{mult-desusp}, $G$ is
stable if and only if the linear coefficient of $\gamma_G$ is equal to 1 - cf. \cite[Proposition 3.8]{vishik19}.

For a commutative ring $R$ define the operator
\[
\Phi=\Phi_R:R[[x]]\to R[[x]],\quad \Phi(G)=(x-1)\diff{G}{x},
\]

\subsection{Stable operations in $\CK$ over $\widehat\Z$}

Recall that in the case $A^*=B^*=\CK^*_{\hat{\zz}}$, the group $\Op_{\widehat\Z}^{n,n}$ of additive operations for $n\leq 0$ and $n\geq 1$ can be identified with $\widehat{\zz}[[x]]$, respectively,
$x\widehat{\zz}[[x]]$.

\begin{proposition}\label{desusp}
 \label{desus-CK}
The desuspension operator $\Sigma^{-1}:\Op_{\widehat\Z}^{n,n}\to \Op_{\widehat\Z}^{n-1,n-1}$ is given by the rule
\[
\Sigma^{-1}(G)=
\left\{
  \begin{array}{ll}
   \Phi(G), & \hbox{if $n\leq 1$;} \\
   \partial^0(\Phi(G))=\Phi(G)-\Phi(G)(0), & \hbox{if $n>1$.}
  \end{array}
\right.
\]
\end{proposition}

\begin{proof}
The Adams operation $\Psi_k$ is identified
with the power series $A_k(x)=(1-x)^k$ if $n\leq 0$ and with $(1-x)^k-1$ if $n>0$.
By Proposition \ref{mult-desusp},
$\Sigma^{-1}\Psi_k=k\Psi_k$, so the formula holds for $G=\Psi_k$.

The map $\Sigma^{-1}$ is continuous in $\tau_o$ and the map $\Phi$ is continuous in $\tau_s$.
Hence both maps are continuous as the maps $\tau_s \to \tau_o$. Since $\tau_o$ is Hausdorff (as $\tau_w$ is), it follows that the set of power series
where $\Sigma^{-1}$ and $\Phi$ coincide is closed in $\tau_s$.
But the set of linear combinations of Adams operations is everywhere dense in $\tau_s$.
\end{proof}

It follows from Proposition \ref{desusp} that the desuspension map $\Sigma^{-1}$ is injective
and yields a tower of injective maps in the other direction:
\begin{equation}\label{reverse1}
\widehat\Z[[x]]=\Op_{\widehat\Z}^{0,0}\xleftarrow{\Sigma^{-1}} \Op_{\widehat\Z}^{1,1}\xleftarrow{\Sigma^{-1}} \ldots\xleftarrow{\Sigma^{-1}} \Op_{\widehat\Z}^{n,n}\xleftarrow{\Sigma^{-1}}\ldots.
\end{equation}
Moreover, the group $\Op_{\widehat\Z}^{st}$ of homogeneous
degree $0$ stable operations $\CK_{\widehat\Z}^*\to \CK_{\widehat\Z}^*$ that is the limit of the sequence \ref{reverse1}
is naturally isomorphic to the group
\[
S:=\cap_{n}\Im(\Phi^n)= \cap_{n}\Im(({\Sigma^{-1}})^n)\subset \widehat\Z[[x]].
\]
Indeed, if $\{G^{(n)}|n\geq 0\}$ is a stable operation, then $G^{(0)}=\Phi^n(G^{(n)})$ for every $n$, hence $G^{(0)}\in S$.
Conversely, given $G\in S$,
write $G=\Phi^n(H^{(n)})$ for every $n$. Since $\Ker(\Phi^n)$ consists of constant power series only, the sequence
$G^{(n)}=\Phi(H^{(n+1)})$ is a stable operation.

\begin{lemma}\label{ifandonlyif}
Let $G\in x\widehat\Z[[x]]$ and $n\geq 1$. Then
\begin{enumerate}
  \item $\partial^n(G)$ has coefficients in $\Z$ if and only if $\partial^{n-1}(\Phi(G))$
has coefficients in $\Z$.
  \item $v(\partial^n(G))\geq m$ for some $m$ if and only if $v(\partial^{n-1}(\Phi(G)))\geq m-1$.
\end{enumerate}
\end{lemma}

\begin{proof}
$(\Rightarrow)$ Follows from Proposition \ref{aformula} for both $(1)$ and $(2)$.

$(\Leftarrow)$ Simply write $H_{k}$ for $(x-1)^k\diff[k]{G}{x}$.
We claim that $\partial^{n-1}(H_{k})$ has coefficients in $\Z$ in case $(1)$
and $v(\partial^{n-1}(H_{k}))\geq m-k$ in case $(2)$ for every $k\geq 1$. We
prove the statements by induction on $k$.

\noindent $(k\Rightarrow k+1)$
We have $H_{k+1}=\Phi(H_{k})- kH_{k}$, hence
\[
\partial^{n-1}(H_{k+1})=\partial^{n-1}\big(\Phi(H_{k})\big)- k\partial^{n-1}(H_{k}).
\]
Then $k\partial^{n-1}(H_{k})$ has coefficients in $\Z$ in case $(1)$ and
$v(k\partial^{n-1}(H_{k}))\geq m-k$ in case $(2)$ by the induction hypothesis.
As the derivative $\partial^{n}(H_{k})$
has coefficients in $\Z$  in case $(1)$ and $v(\partial^{n}(H_{k}))\geq m-k$ in case $(2)$,
it follows from Proposition \ref{aformula}, applied to the power series $H_{k}$,
that $\partial^{n-1}(\Phi(H_{k}))$ also has coefficients in $\Z$ in case $(1)$
and $v\big(\partial^{n-1}(\Phi(H_{k}))\big)\geq m-k-1$ in case $(2)$.
It follows that $\partial^{n-1}(H_{k+1})$ has coefficients in $\Z$ in case $(1)$ and $v(\partial^{n-1}(H_{k+1}))\geq m-k-1$ in case $(2)$.
The claim is proved.

Note that all coefficients of $H_{k}$ are divisible by $k!$ in $\widehat \Z$. It follows that
the power series $\frac{1}{k!}\partial^{n-1}(H_{k})$ have coefficients in $\Z$ in case $(1)$.
By Proposition \ref{aformula}, $\partial^n(G)$ has coefficients in $\Z$ in case $(1)$ and $v(\partial^n(G))\geq m$ in case $(2)$.
\end{proof}

In particular, we can describe the integral operations
$\Op^{n,m}$ as follows.

\begin{proposition}
\label{PhiPnm}
Let $G\in x\widehat\Z[[x]]$ and $m\geq n\geq 1$. Then $G\in\Op^{n,m}$ if and only if
$\Phi^n(G)\in\Z[[x]]$ and $v(\Phi^n(G))\geq m-n$.
\end{proposition}

\begin{proof}
Theorem \ref{mainone} and iterated application of Lemma \ref{ifandonlyif} show that $G\in\Op^{n,m}$ if and only if
$\partial^0(\Phi^{n-1}(G))\in\Z[[x]]$ and $v(\partial^0(\Phi^{n-1}(G)))\geq m-n+1$. Thus, it suffices to prove the following
for a power series $H\in \widehat\Z[[x]]$ and integer $k\geq 0$:

\smallskip
\noindent 1. $\partial^0(H)\in\Z[[x]]$ $\Longleftrightarrow$ $\Phi(H)\in\Z[[x]]$,

\noindent 2. $v(\partial^0(H))\geq k+1$ $\Longleftrightarrow$ $v(\Phi(H))\geq k$.

If $\partial^0(H)\in\Z[[x]]$,
then clearly $\Phi(H)\in\Z[[x]]$. Conversely, if $\Phi(H)\in\Z[[x]]$, then
$\partial^0(H)\in \Q[[x]]\cap \widehat\Z[[x]]=\Z[[x]]$. The second statement follows from the obvious equality
$v(\partial^0(H))=v(\Phi(H))+1$.
\end{proof}

Let $m$ a positive integer.
It follows from Lemma \ref{ifandonlyif}(2)
that there is a tower of inclusions as in (\ref{reverse1}):

\begin{equation}\label{reverse}
x^m\widehat\Z[[x]]=\Op_{\widehat\Z}^{0,m}\hookleftarrow \Op_{\widehat\Z}^{1,m+1}\hookleftarrow \ldots\hookleftarrow \Op_{\widehat\Z}^{n,n+m}\hookleftarrow \ldots.
\end{equation}
and for every $n$ the intersection of $\Op_{\widehat\Z}^{n,n+m}$ and $\Op_{\widehat\Z}^{n+1,n+1}$ in $\Op_{\widehat\Z}^{n,n}$
coincides with $\Op_{\widehat\Z}^{n+1,n+m+1}$. Therefore,
we obtain:

\begin{proposition}
The group of homogeneous
degree $m$ stable operations $\CK_{\widehat\Z}^*\to \CK_{\widehat\Z}^{*+m}$ is naturally isomorphic to
the intersection $x^{\op{max}(0,m)}\widehat\Z[[x]]\cap S$.
\end{proposition}

The map $\Phi: \widehat\Z[[x]]\to \widehat\Z[[x]]$ is continuous in $\tau_w$ and the space
$\widehat\Z[[x]]$ is compact Hausdorff. Hence $\Im(\Phi^n)$ is closed in $\widehat\Z[[x]]$ for any $n$.
It follows that the set $S$ is also closed in $\widehat\Z[[x]]$ in the topology $\tau_w$ and hence in $\tau_o$ and $\tau_s$.

It follow from Proposition \ref{operoper} and Lemma \ref{ifandonlyif} that the topology on $\Op_{\widehat\Z}^{st}$ induced by $\tau_o$ is generated by the neighborhoods of zero $W_m$ consisting
of all collections $\{G^{(n)}|n\geq 0\}$ such that $G^{(n)}$ acts trivially on varieties of dimension $<n+m$.
We still denote this topology by $\tau_o$.

Let $A_r(x)=(1-x)^r\in\widehat\Z[[x]]$ for $r\in\widehat\Z$.
Note that $\Phi(A_r)=r\cdot A_r$.
In particular, if $r$ is invertible in $\widehat\Z$, then $A_r\in S$.

We can describe the set $S$ via divisibility conditions on the coefficients of the power series.

\begin{theorem}\label{desc}
The set $S=\cap_{r}\Im(\Phi^r)\subset \widehat \Z[[x]]$
consists of all power series $G=\Sum_{i\geq 0}a_ix^i$ satisfying the following property: for every prime $p$ and
every positive integers $n$ and $m$
such that $m$ is divisible by $p^n$, for every nonnegative $j<m$ divisible by $p$,
the sum $\Sum_{i=j}^{m-1}\binom{i}{j}a_i$ is divisible by $p^n$.
\end{theorem}

\begin{proof}
Let $n$ be a positive integer, $G\in S$ and write $G=\Phi^n(H)$ for some $H\in \widehat\Z[[x]]$.
Consider the ideal $I=(p^n, x^m)\subset \widehat\Z[[x]]$, where $m$ is divisible by $p^n$. Note that $\Phi(I)\subset I$
since $p^n$ divides $m$.

Let $G'$ be the $x^{m}$-truncation of $G$ and
$H'$ the $x^{m}$-truncation of $H$. As $G-G'\in I$ and $H-H'\in I$, we have $G'-\Phi^n(H')\in I$. Since $G'$ and $\Phi^n(H')$ are polynomials of degree less than $m$,
we conclude that $G'$ and $\Phi^n(H')$ are congruent modulo $p^n$.

We write $G'$ and $H'$ as polynomials in $y=x-1$.
Since $\Phi^n(y^i)=i^n y^i$, the $y^i$-coefficients of $\Phi^n(H')$ are divisible by $p^n$
for all $i$ divisible by $p$. It follows that the same property holds for $G'$. As
\[
G'=\Sum_{i=0}^{m-1} a_i x^i=\Sum_{i=0}^{m-1} a_i (y+1)^i=\Sum_{i=0}^{m-1} a_i \Sum_{j=0}^{i}\mybinom[0.8]{i}{j}y^j=\Sum_{j=0}^{m-1} y^j  \Sum_{i=j}^{m-1}\mybinom[0.8]{i}{j}a_i,
\]
the divisibility condition holds.

Conversely, as $\widehat \Z=\prod\Z_p$, it suffices to prove the statement
over $\Z_p$. Let $G\in\Z_p[[x]]$ satisfy the divisibility condition in the theorem. Choose
$n$ and $m$ such that $m$ is divisible by $p^n$ and set $I=(p^n, x^m)\subset \Z_p[[x]]$ as above. Recall that
$\Phi(I)\subset I$. Let $F$ be the $x^{m}$-truncation of $G$. By assumption, we can write $F\equiv \sum b_i y^i$
modulo $p^n$, where the sum is taken over $i<m$ that are prime to $p$. In particular, $G\equiv \sum b_i y^i$ modulo $I$.

Choose $r>0$ and set $F'=\sum \frac{b_i}{i^r} y^i$. Then $\Phi^r(F')=\sum b_i y^i\equiv G$ modulo $I$, i.e.,
$G$ is in the image of $\Phi^r$ modulo $I$. As $\Im(\Phi^r)$ is closed in $\Z_p[[x]]$ in the topology $\tau_w$, we have $G\in \Im(\Phi^r)$
for all $r$, i.e., $G\in S$.
\end{proof}

\subsection{Stable operations in $\CK$ over $\Z$}

Now we turn to the study of stable operations over $\Z$.

\begin{proposition}
The pre-image of $\Op^{n,n}$ under $\Sigma^{-1}:\Op_{\widehat\Z}^{n+1,n+1}\to \Op_{\widehat\Z}^{n,n}$
is equal to $\Op^{n+1,n+1}$ for every $n\geq 0$.
\end{proposition}

\begin{proof}
As $\Sigma^{-1}=\partial^0\circ\Phi$, for $n\geq 1$, and
$\Sigma^{-1}=\Phi$, for $n=0$, this follows immediately from
Proposition \ref{PhiPnm}.
\end{proof}

Thus, we have a tower
\[
\Z[[x]]=\Op^{0,0}\hookleftarrow \Op^{1,1}\hookleftarrow\ldots\hookleftarrow\Op^{n,n}\hookleftarrow\ldots,
\]
given by the desuspension and the group $\Op^{st}$ of stable homogeneous degree $0$ integral operations is identified with $S_0:=S\cap \Z[[x]]$, where
$S$ is described by Theorem \ref{desc}.
Applying Proposition \ref{PhiPnm} again we get:

\begin{proposition}
 The group of homogeneous
degree $m$ stable operations $\CK^*\to \CK^{*+m}$ is naturally isomorphic to
the intersection $x^{\op{max}(0,m)}\Z[[x]]\cap S_0$.
\end{proposition}

We would like to determine the structure of $S_0$.

\begin{lemma}\label{vandermonde}
For every $n\geq 0$ there is a positive integer $d$ such that $dx^{n}\in S+x^{n+1}\widehat\Z[[x]]$.
\end{lemma}

\begin{proof}
Choose distinct elements $r_0,\ldots, r_n\in \widehat \Z^\times$ such that $r_i-r_j\in\Z$ for all $i$ and $j$.
The $x^i$-coefficients with $i=0,1,\ldots, n$ of the power series $A_{r_j}(x)=(1-x)^{r_j}\in S$ form an $(n+1)\times (n+1)$
Van der Monde type matrix $\big[(-1)^i\binom{r_j}{i}\big]$. Its determinant $d$ is a nonzero integer since all $r_i-r_j$ are integers.
It follows that there is a $\widehat \Z$-linear combination of the $A_{r_j}$'s that is equal to $dx^{n}$ modulo $x^{n+1}$.
\end{proof}

Note that any ideal in $\widehat \Z$ that contains a non-zero integer is generated by a positive integer (the smallest
positive integer in the ideal).
It follows from Lemma \ref{vandermonde} that for every $n\geq 0$ there exists a unique positive integer $d_n$
such that the ideal of all $a\in\widehat\Z$ with the property $ax^{n}\in S+x^{n+1}\widehat\Z[[x]]$ is generated by $d_n$.
We will determine the integers $d_n$ below.

For every $n\geq 0$ choose a power series $G_n\in S$ such that $G_n\equiv d_n x^{n}$ modulo $x^{n+1}$.

\begin{lemma}\label{induct}
Let $G=\sum_{i\geq 0}a_i x^i\in S$ be such that $a_0,\ldots, a_{n-1}\in\Z$. Then there exist $b_i\in\widehat\Z$ for all $i\geq n$
such that $G-\sum_{i\geq n}b_i G_i\in S_0$.
\end{lemma}

\begin{proof}
Find an integer $a'_{n}$ such that $a_{n}-a'_{n}$ is divisible by $d_n$, thus, $a_n=a'_{n}+d_n b_n$ for some $b_n\in\widehat\Z$.
Then the $x^i$-coefficients of $G-b_nG_n$ are integer for $i=0,\ldots, n$. Continuing this procedure, we determine
all $b_i$ for $i\geq n$, so that all coefficients of $G-\sum_{i\geq n}b_i G_i$ are integers.
\end{proof}

\begin{theorem}\label{desrib}
For all $n\geq 0$ there are power series $F_n\in S_0$ such that $F_n\equiv d_n x^{n}$ modulo $x^{n+1}$. Moreover,
\begin{itemize}
\item[$(1)$]
The group $S_0$ consists of all infinite linear combinations $\Sum_{n\geq 0} a_n F_n$ with $a_n\in\Z$.
\item[$(2)$]
The group of homogeneous
degree $m$ stable operations $\CK^*\to \CK^{*+m}$ is naturally
isomorphic to the group
of all infinite linear combinations
$\Sum_{n\geq \op{max}(0,m)} a_n F_n$ with $a_n\in\Z$.
\end{itemize}
\end{theorem}

\begin{proof}
 Fix an $n\geq 0$. The coefficient $d_n$ of $G_n$ is integer. Applying Lemma \ref{induct}, we find $b_i\in\widehat\Z$ for $i\geq n+1$
such that $F_n:=G_n-\sum_{i\geq n+1}b_i G_i\in S_0$. Statements (1) and (2) are clear.
\end{proof}

\subsection{The integers $d_n$}\label{integers}

Our next goal is to determine the integers $d_n$.
Let $n> 0$ be an integer. For an integer $r$ write $L_r$ for the $n$-tuple of binomial coefficients:
\[
\Big(\mybinom[0.8]{r}{0},\mybinom[0.8]{r}{1},\ldots, \mybinom[0.8]{r}{n-1}\Big)=(1,r,\ldots)\in\Z^{n}.
\]

For a $n$-sequence $\bar a=(a_1,\ldots, a_n)$ of positive integers let $d(\bar a)$ be
the determinant of the $n\times n$ matrix with columns $L_{a_1},L_{a_2},\ldots, L_{a_n}$. We have
\begin{equation}\label{formm}
d(\bar a)=\big(\Prod_{s>t} (a_s-a_t)\big)/\Prod_{k=1}^{n-1} k!\in\Z.
\end{equation}

Let $p$ be a prime integer. An $n$-sequence $\bar a$ is called \emph{$p$-prime} if all its terms are prime to $p$.
Let $\bar a_{\min}^{(n)}$ be the ``smallest" strictly increasing $p$-prime $n$-sequence
\[
(1,2,\ldots,p-1,p+1,\ldots).
\]

\begin{lemma}\label{divivi}
Let $\bar a$ be a $p$-prime $n$-sequence that differs from $\bar a_{\min}^{(n)}$ at one term only.
Then $d(\bar a_{\min}^{(n)})$ divides $d(\bar a)$ in the ring of $p$-adic integers $\Z_p$.
\end{lemma}

\begin{proof}
Suppose $\bar a$ is obtained from $\bar a_{\min}^{(n)}$ by replacing a term $a$ by $b$.
It follows from (\ref{formm}) that
\[
d(\bar a)/d(\bar a_{\min}^{(n)})=\Prod(b-a')/\Prod(a-a'),
\]
where the products are taken over all terms $a'$ of $\bar a_{\min}$ but $a$.
Since $a$ is prime to $p$, the product $\Prod(a-a')$ generates the same ideal in $\Z_p$
as $a!(c-a)!$, where $c$ is the last term of $\bar a_{\min}$. Similarly, as $b$ is prime to $p$,
the product $\Prod(b-a')$ generates the same ideal in $\Z_p$ as
\[
b(b-1)\cdots (b-a+1)\cdot (b-a-1)\cdots (b-c+1)(b-c)
=(-1)^{c-a}a!(c-a)!\mybinom[0.8]{b}{a}\mybinom[0.8]{c-b}{c-a}.\qedhere
\]
\end{proof}

\begin{corollary}\label{inzp}
The integer $d(\bar a_{\min}^{(n)})$ divides $d(\bar a_{\min}^{(n+1)})$ in $\Z_p$.
\end{corollary}

\begin{proof}
In the cofactor expansion (Laplace's formula) of the determinant $d(\bar a_{\min}^{(n+1)})$ along the last row
all minors are divisible by $d(\bar a_{\min}^{(n)})$ in view of Lemma \ref{divivi}.
\end{proof}

Write $M_n$ for the $\Z_p$-submodule of $(\Z_p)^n$
generated by the tuples $L_{a_1},L_{a_2},\ldots, L_{a_n}$, where $(a_1,a_2,\ldots,a_n)=\bar a_{\min}^{(n)}$.

\begin{lemma}\label{divivi2}
Let $b$ be an integer prime to $p$. Then the $n$-tuple $L_b$ is contained in $M_n$.
In others words, the $\Z_p$-submodule of $(\Z_p)^n$ generated by $L_{b}$ for all integers $b>0$ prime to $p$
coincides with $M_n$.
\end{lemma}

\begin{proof}
By Cramer's rule, the solutions of the equation $L_b=x_1L_{a_1}+\ldots + x_n L_{a_n}$ are given by the formula
$x_i=d(\bar a_{(i)})/d(\bar a_{\min}^{(n)})$, where the sequence $\bar a_{(i)}$ is obtained from $\bar a_{\min}^{(n)}$
by replacing the $i$-th term with $b$. By Lemma \ref{divivi}, we have $x_i\in\Z_p$.
\end{proof}

The following statement is a generalization of Lemma \ref{divivi}.

\begin{corollary}\label{allall}
Let $\bar a$ be any $p$-prime $n$-sequence.
Then $d(\bar a_{\min}^{(n)})$ divides $d(\bar a)$ in $\Z_p$. \qed
\end{corollary}

Set
\[
d_n=d_n^{(p)}:=d(\bar a_{\min}^{(n+1)})/d(\bar a_{\min}^{(n)}).
\]
By Corollary \ref{inzp}, $d_n\in\Z_p$.

Write $n$ in the form $n=(p-1)k+i$, where $i=0,1\ldots, p-2$ and $k=\floor{\frac{n}{p-1}}$. Then
it follows from (\ref{formm}) that
\begin{equation}\label{padic}
d_n\Z_p=\frac{p^k\cdot k!}{n!}\Z_p,\quad\text{or, equivalently}\quad v_p(d_n)=k+v_p(k!)-v_p(n!),
\end{equation}
where $v_p$ is the $p$-adic discrete valuation.

Note that $v_p((n+k)!)=v_p((pk+i)!)=v_p((pk)!)=k+v_p(k!)$, hence $d_n\Z_p=\frac{(n+k)!}{n!}\Z_p$.
Observe that the function $n\mapsto v_p(d_n)$ is not monotonic.

\begin{proposition}\label{mama}
An $(n+1)$-tuple $(0,0,\ldots,0,d)$ is contained in $M_{n+1}$ if and only if $d$ is divisible by $d_n$ in $\Z_p$.
\end{proposition}

\begin{proof}
As in the proof of Lemma \ref{divivi2}, $(0,0,\ldots,0,d)\in M_{n+1}$ if and only if $d\cdot d(\bar a_{(i)})$
is divisible by $d(\bar a_{\min}^{(n+1)})$
in $\Z_p$ for all $i$, where the sequence $\bar a_{(i)}$ is obtained from $\bar a_{\min}^{(n)}$
by deleting the $i$-th term in $\bar a_{\min}^{(n+1)}$. We have $d(\bar a_{(i)})=d(\bar a_{\min}^{(n)})$ if $i=n+1$ and by Corollary \ref{allall}, all $d(\bar a_{(i)})$ are divisible by $d(\bar a_{\min}^{(n)})$, whence the result.
\end{proof}

For an integer $r$, let as before $A_r(x)=(1-x)^r$. Note that the $n$-tuple $L_r$ is the tuple
of coefficients (after appropriate change of signs) of the $x^n$-truncation of the polynomial $A_r$.
Denote by $N^{(p)}$ the $\Z_p$-submodule of $\Z_p[x]$ generated by $A_r$ for all integers $r>0$ prime to $p$.
We get an immediate corollary from
Lemma \ref{divivi2} and Proposition \ref{mama}:

\begin{proposition}\label{poly}
Let $d\in\Z_p$ and $n\geq 0$. Then $dx^n\in N^{(p)}+x^{n+1}\Z_p[x]$ if and only if $d$ is divisible by $d_n$.
Moreover, there is a $\Z_p$-linear combination $G_n$ of the Adams polynomials
$A_{a_1},A_{a_2},\ldots, A_{a_{n+1}}$, where $(a_1,a_2,\ldots,a_{n+1})=\bar a_{\min}^{({n+1})}$, such that
$G_n \equiv  d_n x^n\ (mod\ x^{n+1})$.\qed
\end{proposition}

\begin{proposition}\label{poly2}
The set $S_{\Z_p}=\cap_{r}\Im(\Phi_{\Z_p}^r)$ contains a power series $\equiv  d x^n\ (mod\ x^{n+1})$ if and only if
$d$ is divisible by $d_n$ in $\Z_p$.
\end{proposition}

\begin{proof}
Suppose that $G\in S_{\Z_p}$ and $G\equiv dx^n$ modulo $x^{n+1}$. Choose integers $k>0$
such that $p^k$ is divisible by $d(\bar a_{\min}^{(n+1)})$
and $m>n$ divisible by $p^k$ and consider the ideal $I=(p^k, x^m)\subset \Z_p$. We have $G=\Phi^k(G')$ for some $G'\in \Z_p[[x]]$ and write
\[
G'=\Sum_{i=0}^{m-1}b_i A_{i} \quad\text{modulo}\quad x^{m}\Z_p[[x]]
\]
for some $b_i\in\Z_p$. Applying $\Phi^k$ and taking into account the equality $\Phi^k(A_{i})=i^k A_{i}$, we get
$G=\Phi^k(G')\in N^{(p)}+I$. Taking the $x^{n+1}$-truncations, we see that
\[
dx^n\in N^{(p)} +x^{n+1}\Z_p[x]+p^k\Z_p[x].
\]
As $p^k$ is divisible by $d(\bar a_{\min}^{(n+1)})$, we conclude that $p^k\Z_p[x]\subset N^{(p)} +x^{n+1}\Z_p[x]$, hence
$dx^n\in N^{(p)} +x^{n+1}\Z_p[x]$. By Proposition \ref{poly}, $d$ is divisible by $d_n$.
\end{proof}


Now we turn to the ring $\widehat\Z$. The integers $d_n$ defined before Lemma \ref{induct} are the products of
primary parts of $d_n=d_n^{(p)}$ determined as above for every prime $p$. In view of (\ref{padic}) we have
\[
v_p(d_n)=\floor{\frac{n}{p-1}}+v_p(\floor{\frac{n}{p-1}}!)-v_p(n!)
\]
for every prime $p$. For example,
$d_0 =1$, $d_1 =2$, $d_2 =2^2\cdot 3$, $d_3 =2^3$, $d_4 =2^4\cdot 3\cdot 5$, $d_5 =2^5\cdot 3$,
$d_6 =2^6\cdot 3^2\cdot 7$, $d_7 =2^7\cdot 3^2$.

Propositions \ref{poly} and \ref{poly2} yield:

\begin{theorem}\label{poly3}
Let $n\geq 0$ be an integer. Then
\begin{enumerate}
  \item There is a $\widehat\Z$-linear combination $G_n$ of the Adams polynomials
$A_{a_1},A_{a_2},\ldots, A_{a_{n+1}}$ for some $a_1,a_2,\ldots, a_{n+1}\in \widehat\Z^\times$ such that $G_n\equiv d_nx^n$ modulo $x^{n+1}$.
  \item The set $S=\cap_{r}\Im(\Phi_{\widehat\Z}^r)$ contains a power series $\equiv  d x^n\ (mod\ x^{n+1})$
if and only if $d$ is divisible by $d_n$ in $\widehat\Z$. It consists of all (infinite) linear combinations of $G_n$.
\end{enumerate}
\end{theorem}

\begin{remark}
It follows from Proposition \ref{poly} that $a_1,a_2,\ldots, a_{n+1}\in \widehat\Z^\times$ can be chosen so that
for every prime $p$, we have $((a_1)_p,(a_2)_p,\ldots,(a_{n+1})_p)=\bar a_{\min}^{({n+1})}$ with respect to $p$.
In particular, $a_1=1$.
\end{remark}

\begin{proposition}
 \label{taus-zamykanie}
 The set $S=\cap_r\op{Im}(\Phi^r)$ is the closure in the topology
 $\tau_s$, and hence, in the topologies $\tau_o$ and $\tau_w$ of the set of all
 (finite) $\widehat{\zz}$-linear combinations of the power series
 $A_r$ for $r\in\widehat{\zz}^{\times}$.
\end{proposition}

\begin{proof}
 Denote as $T_s,T_w,T_o$ the closures of the mentioned set of linear combinations in our three topologies.
As $S$ is closed in $\tau_w$, we have $T_s\subset T_o \subset T_w \subset S$.

 Let $G\in x^k\widehat{\zz}[[x]]\cap S$. Then by Theorem \ref{poly3},
 $G\equiv dx^k\,\,(mod\,x^{k+1})$, where $d=d_k\cdot c$, for
 some $c\in\widehat{\zz}$. We know that there exists a
 $\widehat{\zz}$-linear combination $G_k$ of the power series
 $A_{a_1}, A_{a_2},\ldots,A_{a_k}$ (with invertible $a_i$'s)
 such that
 $G_k\equiv d_kx^k\,\,(mod\,x^{k+1})$. Hence,
 $G-c\cdot G_k\in x^{k+1}\widehat{\zz}[[x]]\cap S$. Applying this inductively,
 we obtain that, for any $G\in S$ and any
 positive integer $m$, there exists a finite
 $\widehat{\zz}$-linear combination $H$ of invertible $A_r$'s, such that $G-H\in x^m\widehat{\zz}[[x]]\cap S$.
 Therefore, $T_s=S$ and hence $T_s=T_o=T_w=S$.
\end{proof}

\subsection{Stable operations in $K_{gr}$}
In section \ref{graded} we defined the bi-algebra $A$ of co-operations in $K_{gr}$ with a canonical element $s\in A$.
Recall that for a commutative ring $R$, the bi-algebra of operations $\Op^{n,n}_R(K_{gr})=\Op^{0,0}_R(\CK)=R[[x]]$
is dual to $A$. The same proof as in Proposition \ref{desusp} shows that the desuspension operator
\[
\Sigma^{-1}:R[[x]]=\Op_{R}^{n,n}(K_{gr})\to\Op_{R}^{n-1,n-1}(K_{gr})=R[[x]]
\]
coincides with $\Phi$. It follows that
\[
\Op_{R}^{st}(K_{gr})=\lim(R[[x]] \xleftarrow{\Phi} R[[x]] \xleftarrow{\Phi} R[[x]] \xleftarrow{\Phi} \ldots).
\]

\begin{lemma}
The desuspension operator $\Phi$ is dual to the multiplication by $s$ in $A$.
\end{lemma}

\begin{proof}
As $\Phi((1-x)^m)=m(1-x)^m$, in view of Lemma \ref{vish} we have
\[
\langle e_n,\Phi((1-x)^m) \rangle=\langle e_n,m(1-x)^m \rangle=m\cdot e_n(m)=\langle se_n,(1-x)^m \rangle.\qedhere
\]
\end{proof}

The localization $A[\frac{1}{s}]$ can be identified with $\Colim(A \xra{s} A\xra{s}  \ldots)$. Therefore,
\[
\Op_{R}^{st}(K_{gr})\simeq \Hom(A\Big[\frac{1}{s}\Big], R),
\]
i.e., the bi-algebra $\Op_{R}^{st}(K_{gr})$ of stable operations is dual to $A[\frac{1}{s}]$.

The bi-algebra $A[\frac{1}{s}]$ coincides with the algebra of degree $0$ stable operations $K_0(K)$ in topology
(see \cite[Proposition 3]{Johnson84} and \cite{AHS71}). Moreover, $A[\frac{1}{s}]$ is a free abelian group
of countable rank \cite[Theorem 2.2]{AC77} and can be described as the set of all Laurent polynomials
$f\in\Q[s,s^{-1}]$ such that $f(\frac{a}{b})\in\Z[\frac{1}{ab}]$ for all integers $a$ and $b\neq 0$.

It follows that the bi-algebra $A[\frac{1}{s}]$ admits an antipode $s\mapsto s^{-1}$ that makes $A[\frac{1}{s}]$ a Hopf algebra.
It follows that $\Op_{R}^{st}(K_{gr})$ is a (topological) Hopf algebra.

\begin{remark}
We have a diagram of homomorphisms of bi-algebras and its dual:
\[
\xymatrix{
\Z[s]  \ar[d] \ar[r] & A  \ar[d]  & & & R^{[0,\infty)}    & R[[x]] \ar[l]_{b} \\
\Z[s,s^{-1}] \ar[r]    &  A[\frac{1}{s}]     & & &  R^{\Z} \ar[u]     &  \Op_{R}^{st}(K_{gr}) \ar[u]\ar[l]
}
\]
The bottom maps are homomorphisms of Hopf algebras. The antipode of $R^{\Z}$ takes a sequence $r_i$ to $r_{-i}$.
\end{remark}

The group of degree $0$ stable operations $\Op_{\widehat\Z}^{st}(K_{gr})$ coincides with $\Op_{\widehat\Z}^{st}(\CK)=S$ whose structure
was described in Theorem \ref{poly3}. Our nearest goal is to
determine the structure of $\Op_{\Z}^{st}(K_{gr})$.
We remark that this group is different from $\Op_{\Z}^{st}(\CK)=S\cap\Z[[x]]$.

Let $R$ be one of the following rings: $\Z$, $\Z_p$ or  $\widehat\Z$. Recall that we have an injective homomorphism
$b_R:R[[x]]\to R^{[0,\infty)}$ taking $(1-x)^m$ to the sequence $(1,m,m^2,\ldots)$. The operation $\Phi$ on $R[[x]]$
corresponds to the shift operation $\Pi$ on $R^{[0,\infty)}$ defined by $\Pi(a)_i=a_{i+1}$.

An \emph{$n$-interval} of a sequence $a$ in $R^{[0,\infty)}$ or $R^{\Z}$ is the $n$-tuple $(a_i, a_{i+1},\ldots, a_{i+n-1})$
for some $i$. We say that this interval \emph{starts at $i$}.

For every $n\geq 1$, let $M_n$ be the $R$-submodule of $R^n$ generated by the $n$-tuples
$\bar r:=(1,r,r^2,\ldots, r^{n-1})$
for all integers $r>0$. Note that $M_n$ is of finite index in $R^n$.

\begin{lemma}\label{seqq}
A sequence $a\in R^{[0,\infty)}$ belongs to the image of $b_R$ if and only if
for every $n>0$, the $n$-interval of $a$ starting at $0$ is contained in $M_n$.
\end{lemma}

\begin{proof}
The implication $\Rightarrow$ is clear. For the converse note that by assumption $a$ is contained in
the closure of $\Im(b_R)$. On the other hand, if $R=\Z_p$ or $\widehat \Z$, the space
$R[[x]]$ is compact in $\tau_w$ and $R^{[0,\infty)}$ is Hausdorff, hence $\Im(b_R)$ is closed, i.e., $a\in \Im(b_R)$.
If $R=\Z$, it follows from
the case $R=\widehat \Z$ that $a=b_{\widehat \Z}(G)$ for some $G\in \widehat \Z[[x]]$. Since at the same time $G \in  \Q[[x]]$,
we have $G\in \Z[[x]]$.
\end{proof}


Let $T_R\subset R^{\Z}$ be the $R$-submodule of all sequences $a\in R^{\Z}$ such that every $n$-interval of $a$
is contained in $M_n$ for all $n\geq 1$. If $a\in T_R$, by Lemma \ref{seqq}, for every $n\geq 0$ there is $G_n\in R[[x]]$ such that
$b_R(G_n)=(a_{-n}, a_{-n+1},\ldots)$. Since $\Phi(G_{n+1})=G_n$, the sequence $(G_n)_{n\geq 0}$ determines an element
in $\Op_{R}^{st}(K_{gr})$. This construction establishes an isomorphism $\Op_{R}^{st}(K_{gr})\simeq T_R$.
Note that $T_{\widehat\Z}=\Op_{\widehat\Z}^{st}(\CK)=S=\cap_r\op{Im}(\Phi^r)$.

\smallskip

For every $n\geq 1$, let $N_n$ be the $R$-submodule of $R^n$ generated by the $n$-tuples
$\bar r$ for all $r\in R^\times$. Then $N_n$ is of finite index in $R^n$ if $R=\Z_p$ or  $\widehat\Z$.

Note that every $n$-tuple $\bar r$ with $r\in R^\times$ extends to the sequence $a$ with $a_i=r^i$ that is contained
in $T_R$.

\begin{lemma}\label{subb}
$N_n\subset M_n$ for all $n\geq 1$.
\end{lemma}

\begin{proof}
It suffices to consider the case $R=\Z_p$.
Choose an integer $m>0$ such that $p^m\cdot \Z_p^n\subset M_n$. Let $r\in\Z_p^\times$. Find an integer $r'>0$ congruent to $r$ modulo $p^m$.
Then the tuple $\bar r=(1,r,r^2,\ldots, r^{n-1})$ is congruent to $\bar r'$ modulo $p^m$. Hence
$\bar r=\bar r'+(\bar r-\bar r')\in M_n+p^m \Z_p^n\subset M_n$.
\end{proof}

It follows from Lemma \ref{subb} that every element in $N_n$ is an $n$-interval of a sequence in $T_R$.

\begin{proposition}
If $R=\Z_p$ or $\widehat \Z$, the $R$-module $T_R$ consists of all sequences $a\in R^{\Z}$
such that every $n$-interval of $a$ is contained in $N_n$ for all $n\geq 1$.
\end{proposition}

\begin{proof}
We may assume that $R=\Z_p$.
Let $a\in T_R$. In view of Lemma \ref{subb} it suffices to show that every $n$-interval $v$ of $a$ starting at $i$ is contained in $N_n$ for all $n\geq 1$.
Take an integer $m>0$ and consider the $(n+m)$-interval $w$ of $\Pi^{-m}(a)$ starting at $i$, so that $v$ is the part of $w$ on the right.
Write $w$ as a (finite) linear combination $\sum t_r \bar r$ over positive integers $r$,
where $t_r\in\Z_p$ and $\bar r=(1,r,r^2,\ldots, r^{n+m-1})\in M_{{n+m-1}}$.
Applying $\Pi^{m}$ to $\Pi^{-m}(a)$ we see that $v=\sum t_r r^m \hat r$, where $\hat r=(1,r,r^2,\ldots, r^{n-1})\in M_{{n-1}}$.
As $r^m$ is divisible by $p^m$ if $r$ is divisible by $p$,
it follows from the definition of $N_n$ that $v\in N_n+p^mM_n$. Since $N_n$ is of finite index in $M_n$, we can choose $m$
such that $p^mM_n\subset N_n$, hence $a\in N_n$.
\end{proof}

Denote by $\theta:R^{\Z}\to R^{\Z}$ the reflection operation taking a sequence $a$ to the sequence $\theta(a)_i=a_{-i}$.

\begin{corollary}
The module $T_R$ is invariant under $\theta$.
\end{corollary}

\begin{proof}
In the case $R=\Z_p$ or $\widehat \Z$ it suffice to notice that if $r\in\R^\times$, the symmetric $n$-tuple
$(r^{n-1}, r^{n-2},\ldots, r,1)=r^{n-1}(1,r^{-1},(r^{-1})^2,\ldots, (r^{-1})^{n-1})$ is contained in $N_n$.
If $R=\Z$ the statement follows from the equality $T_{\Z}=T_{\widehat \Z}\cap \Z^{\Z}$.
\end{proof}

Now let $R=\widehat\Z$ and $n\geq 0$.
The ideal of all $t\in\widehat\Z$ such that $(0,\ldots,0,t)\in N_n$ is generated by a
(unique) positive integer $\tilde d_n=n!\cdot d_n$, where the integers $d_n$ were introduced in Section \ref{integers}.
We know that
\[
v_p(\tilde d_n)=v_p((n+k_p)!)
\]
for all primes $p$, where $k_p=\floor{\frac{n}{p-1}}$. By Theorem \ref{desrib},
there are power series $F_n\in S_0=S\cap\Z[[x]]$ such that $F_n\equiv d_n x^{n}$ modulo $x^{n+1}$.

Let $f^{(n)}\in T_{\widehat\Z}$ be the image of $F_n$ under the map
$S=\Op_{\widehat\Z}^{st}(K_{gr})\to {\widehat\Z}^{\Z}$. Thus,
$(0,\ldots,0,\tilde d_n)$ is the $n$-interval of $f^{(n)}$ starting at $0$.
For example, we can choose:
\[
f^{(0)}=(\ldots, 1,1,1,1,\ldots),
\]
\[
f^{(1)}=(\ldots, 0,2,0,2,\ldots).
\]
As in the proof of Theorem \ref{desrib}, modifying $f^{(n)}$ by adding
multiples of the shifts of $f^{(m)}$ for $m>n$ and their reflections
we can obtain $f^{(n)}\in \Z^{\Z}$ for all $n$.

\begin{theorem}
\label{stable-op-Kgr}
Every sequence $a\in T_{\Z}\simeq \Op_{\Z}^{st}(K_{gr})$ can be written in the form
\[
a=\sum_{i=0}^{\infty}\Big[b_{2i}\Pi^{-i}\theta(f^{(2i)})+b_{2i+1}\Pi^{i}(f^{(2i+1)})\Big]
\]
for unique $b_0,b_1,\ldots\in\Z$.
\end{theorem}

\begin{proof}
We determine the integers $b_0,b_1,\ldots$ inductively so that for every $m\geq 0$ the sum $\sum_{i=0}^{m}$ of the terms in
the right hand side and the sequence $a$ have the same $2m+2$-intervals starting at $-m$.
\end{proof}

\begin{remark}
Observe that
$\{\Pi^{-i}\theta (f^{(2i)}),\,\Pi^i (f^{(2i+1)})\,|
\,i\in\zz_{\geq 0}\}$ is also a topological basis of
${\Op}^{st}_{\widehat{\zz}}(K_{gr})$. Note that,
at the same time,
$\{f^{(j)}\,|\,j\in\zz_{\geq 0}\}$ form a topological
basis for ${\Op}^{st}_{\zz}(\op{CK})$ and
${\Op}^{st}_{\widehat{\zz}}(\op{CK})$.
This shows the relation between operations in $\CK$ and those in $K_{gr}$. In particular, that there are substantially more operations in the former theory.
\end{remark}

\subsection{Stable multiplicative operations}

We first consider stable multiplicative operations on $\CK_{\wzz}^*$. From Proposition 5.3 we obtain:

\begin{proposition}
 \label{stable-mult-S}
 Stable multiplicative operations
 $\op{CK}_{\wzz}^*\row\op{CK}_{\wzz}^*$ are exactly operations
 $\Psi^c_1$, for $c\in\wzz^{\times}$.
 These are invertible and form a group isomorphic to
 $\wzz^{\times}$.
 Similarly, stable multiplicative operations on $\op{CK}^*$ form
 a group isomorphic to $\zz^{\times}$.
\end{proposition}

Restricted to $\op{CK}_{\wzz}^0$, the operation $\Psi^c_1$
is given by $G_0=1$ (as it is multiplicative and so, maps $1$ to
$1$), while $G(tx)=G(t)G(x)=ct\cdot\gamma_G(x)=1-(1-tx)^c$ and
so, our operation corresponds to the power series
$A_c=(1-x)^c$. In other words, on $\op{CK}_{\wzz}^0$, operation
$\Psi^c_1$ coincides with the Adams operation $\Psi_c$.
Then on $\op{CK}_{\wzz}^n$ it is equal to $c^{-n}\cdot\Psi_c$.

Proposition \ref{taus-zamykanie} gives:

\begin{corollary}
 \label{stab-add-stab-mult}
 The set of homogeneous stable additive operations on
 $\op{CK}^*_{\wzz}$ is the closure in the topology $\tau_o$ of
 the set of (finite)
 $\wzz$-linear combinations of stable multiplicative operations.
\end{corollary}

\begin{remark}
Note that the respective statement for $\zz$-coefficients is not true, as there are only two stable multiplicative operations on
$\op{CK}^*$, namely, $\Psi^1_1$ and $\Psi^{-1}_1$, while the
group of stable additive operations there has infinite (uncountable) rank.
\end{remark}

Now we consider stable multiplicative operations on $K_{gr}^*$ over $\Z$.

\begin{proposition}
 \label{stable-mult-S0}
 Stable multiplicative operations
 $K_{gr}^*\row K_{gr}^*$ are exactly operations
 $\Psi^c_1$, for $c=\pm 1$.
 These are invertible and form a group isomorphic to $\zz^{\times}\cong\zz/2\zz$.
\end{proposition}

\begin{proof}
The linear coefficient of $\gamma_G$ for the operation $\ppsi{l}{c}{b}$ is $t^{1-l}b$ -
see 5.2. This will be equal to $1$ exactly when
 $l=1$ and $b=1$.
\end{proof}

As above, the operation $\Psi^c_1$ corresponds
to the power series $A_c=(1-x)^c$. On $K^n_{gr}$ it coincides
with $c^{-n}\cdot\Psi^1_c$.


\end{document}